\documentclass[twoside, a4paper,12pt]{amsart}

\usepackage{tikz}
\usepackage{amsfonts}
\usepackage[hmarginratio=1:1]{geometry}

\theoremstyle{plain}
\newtheorem*{claim*}{Claim}
\newtheorem{thm}{Theorem}[section]
\newtheorem{corollary}[thm]{Corollary}
\newtheorem{lemma}[thm]{Lemma}
\newtheorem{prop}[thm]{Proposition}
\theoremstyle{definition}
\newtheorem{defn}[thm]{Definition}
\newtheorem{ex}[thm]{Example}

\newtheorem{remark}[thm]{Remark}
\newtheorem{con}[thm]{Construction}

\newtheorem{prob}[thm]{Open Problem}
\begin{document}
\title{\large{Right noetherian semigroups}}
\author{Craig Miller and Nik Ru{\v s}kuc}
\address{School of Mathematics and Statistics, St Andrews, Scotland, UK, KY16 NSS}
\email{cm380@st-andrews.ac.uk, nik.ruskuc@st-andrews.ac.uk}

\begin{abstract}
A semigroup $S$ is {\em right noetherian} if every right congruence on $S$ is finitely generated.
In this paper we present some fundamental properties of right noetherian semigroups, 
discuss how semigroups relate to their substructures with regard to the property of being right noetherian,
and investigate whether this property is preserved under various semigroup constructions.
\end{abstract}

\maketitle

\section{\large{Introduction}\nopunct}

A semigroup $S$ is said to be {\em right noetherian} if every right congruence on $S$ is finitely generated;
equivalently, $S$ satisfies the ascending chain condition on its right congruences.\footnote{There are other uses of the term `right noetherian' in semigroup theory:
it may be used to refer to semigroups for which every right ideal is finitely generated,
and to semigroups for which there is no infinite descending sequence with respect to the proper right-divisibility relation \cite{Dehornoy}.}
Left noetherian semigroups are defined in an analogous way.
In this paper we will only consider right noetherian semigroups.\par
For groups, being right noetherian means that every subgroup is finitely generated, 
and this is equivalent to being left noetherian.
Groups with this property are commonly known as {\em noetherian} groups and have received a fair amount of attention.
They were first studied systematically by Baer in \cite{Baer},
and it was shown in \cite{Hirsch} that polycyclic groups are noetherian.
Noetherian groups can have a rather complex structure, as explained in Chapter 9 of \cite{Olshanskii}.
In particular, it is shown there that there exist infinite simple noetherian groups that are periodic or, on the other extreme, torsion-free \cite[Theorem 28.1, Theorem 28.3]{Olshanskii}.
However, every noetherian group is finitely generated.\par
The study of right noetherian semigroups was initiated by Hotzel in \cite{Hotzel}.
He obtained particularly comprehensive results regarding weakly periodic semigroups (which include all regular semigroups);
in particular, he showed that if they are right noetherian, then they contain only finitely many right ideals and are finitely generated.
The question of whether every right noetherian semigroup is finitely generated has been considered by several authors (see \cite{Bergman}, \cite{Hotzel}, \cite{Kozhukhov2}) and remains an open problem.
Right noetherian semigroups were further studied by Kozhukhov in \cite{Kozhukhov1},
where he discusses the structure of right noetherian $0$-simple semigroups and right noetherian inverse semigroups.\par
The purpose of this paper is to investigate how the property of being right noetherian behaves under various semigroup constructions.
It is structured as follows.
In Section 2, we provide some fundamental facts about right noetherian semigroups; 
many of these results are folklore, but we gather them for completeness.
In Section 3, we investigate whether the property of being right noetherian is inherited by substructures in certain situations.
The remainder of the paper is concerned with various semigroup constructions.
Specifically, we consider the following:
direct products (Section 4), semidirect and wreath products (Section 5), $0$-direct unions and (strong) semilattices of semigroups (Section 6), and semigroup free products and monoid free products (Section 7).

\section{\large{Fundamental properties}\nopunct}

We begin this section by presenting several standard equivalent characterisations of right noetherian semigroups in terms of the ascending chain condition and maximal condition on right congruences,
paralleling the situation for rings.

\begin{prop}
\label{acc}
Let $S$ be a semigroup.  Then the following are equivalent:
\begin{enumerate}
\item $S$ is right noetherian;
\item $S$ satisfies the {\em ascending chain condition} on its right congruences: 
every infinite ascending chain $\rho_1\subseteq\rho_2\subseteq\dots$ of right congruences on $S$ eventually terminates;
\item every non-empty set of right congruences on $S$ has a maximal element.
\end{enumerate}
\end{prop}

\begin{proof}
$(1)\Rightarrow(2)$.  Consider an infinite ascending chain 
$$\rho_1\subseteq\rho_2\subseteq\dots$$ of right congruences on $S.$
Let $\rho$ be the right congruence $\bigcup_{i=1}^{\infty}\rho_i.$
Since $S$ is right noetherian, we have that $\rho$ is generated by some finite set $X.$
Let $k$ be minimal such that $X\subseteq\rho_k.$
Then $\rho\subseteq\rho_k.$ 
But $\rho_k\subseteq\rho,$ so $\rho=\rho_k.$
Hence, we have that $\rho_n=\rho_k$ for all $n\geq k.$\par
$(2)\Rightarrow(3)$.  Suppose there exists a non-empty set $F$ of right congruences on $S$ with no maximal element.
Choose $\rho_1\in F.$  Since $\rho_1$ is not maximal, there exists $\rho_2\in F$ with $\rho_1\subset\rho_2.$
Since $\rho_2$ is not maximal, there exists $\rho_3\in F$ with $\rho_2\subset\rho_3.$
Continuing in this way, there exists an infinite strictly ascending chain
$$\rho_1\subset\rho_2\subset\rho_3\subset\dots$$
of right congruences on $S.$\par
$(3)\Rightarrow(1)$.  Suppose there exists a non-finitely generated right congruence $\rho$ on $S.$
Choose $(a_1, b_1)\in\rho,$ and let $\rho_1$ be the right congruence on $S$ generated by $(a_1, b_1).$
Since $\rho$ is not finitely generated, we have that $\rho_1\neq\rho.$ 
Choose $(a_2, b_2)\in\rho\setminus\rho_1,$ and let $\rho_2$ be the right congruence on $S$ generated by $\{(a_1, b_1), (a_2, b_2)\}.$
We have that $\rho_1\subset\rho_2.$
Continuing in this way, we have a non-empty set $\{\rho_1, \rho_2, \dots\}$ of right congruences on $S$ with no maximal element.
\end{proof}

Now, let $S$ be a semigroup and let $X\subseteq S\times S$.  We introduce the notation
$$\overline{X}=X\cup\{(x, y) : (y, x)\in X\},$$
which will be used throughout the paper.
For $a, b\in S$, an {\em $X$-sequence connecting} $a$ and $b$ is any sequence
$$a=x_1s_1, \; y_1s_1=x_2s_2, \; y_2s_2=x_3s_3, \; \dots, \; y_ks_k=b,$$
where $(x_i, y_i)\in\overline{X}$ and $s_i\in S^1$ for $1\leq i\leq k.$\par
We now introduce the following definition.

\begin{defn}
Let $S$ be a semigroup, let $X\subseteq S\times S,$ and let $s, t\in S.$
We say that $(s, t)$ is a {\em consequence} of $X$ if either $s=t$ or there exists an $X$-sequence connecting $s$ and $t.$
\end{defn}

We have the following basic lemma.

\begin{lemma}
Let $S$ be a semigroup, let $X\subseteq S\times S,$ let $\rho$ be a right congruence on $S$ generated by $X,$ and let $s, t\in S$.
Then $s\,\rho\,t$ if and only if $(s, t)$ is a consequence of $X$.
\end{lemma}

The following lemma is a specialisation of a well-known fact from general algebra (see \cite[Theorem 2.5.5]{Burris} for more details), and will be used repeatedly throughout the paper.

\begin{lemma}
\label{finite subset}
Let $S$ be a semigroup, let $\rho$ be a finitely generated right congruence on $S,$ and let $X$ be a generating set for $\rho.$
Then there exists a finite set $X^{\prime}\subseteq X$ such that $\rho$ is generated by $X^{\prime}.$
\end{lemma}

\begin{defn}
A semigroup $S$ is said to be {\em weakly right noetherian} if every right ideal of $S$ is finitely generated.
\end{defn}

\begin{remark}
Being weakly right noetherian for a semigroup $S$ is equivalent to $S$ satisfying the ascending chain condition on right ideals.
\end{remark}

\begin{lemma}
\label{weakly}
Let $S$ be a semigroup.  If $S$ is right noetherian, then $S$ is also weakly right noetherian.
\end{lemma}

\begin{proof}
Let $I$ be a right ideal of $S.$
Consider the Rees right congruence $\rho_I$ on $S$ given by
$$a\,\rho_I\,b\iff a=b\text{ or }a, b\in I.$$
Since $S$ is right noetherian, $\rho_I$ is generated (as a right congruence) by some finite set $X\subseteq I\times I.$
We claim that $I$ is generated (as a right ideal) by the finite set 
$$U=\{x\in I : (x, y)\in\overline{X}\text{ for some }y\in I\}.$$
Indeed, let $a\in I,$ and choose $b\in I$ with $a\neq b.$
Since $a\,\rho_I\,b,$ there exists an $X$-sequence connecting $a$ and $b.$
In particular, there exist some $x\in U$ and $s\in S^1$ such that $a=xs.$
\end{proof}

\begin{corollary}\cite[Lemma 1.6]{Hotzel}
\label{weaklycorollary}
Let $S$ be a right noetherian semigroup.  Then $S$ contains no infinite set of pairwise incomparable right ideals.
\end{corollary}

\begin{defn}
Let $S$ be a semigroup.  An element $s\in S$ is said to be {\em decomposable} if $s\in S^2.$
An element is {\em indecomposable} if it is not decomposable.
\end{defn}

\begin{corollary}
\label{indecomposable}
Let $S$ be a right noetherian semigroup.  Then $S$ has only finitely many indecomposable elements.
\end{corollary}

\begin{proof}
Since $S$ is weakly right noetherian, $S$ is finitely generated as a right ideal; 
that is, there exists a finite set $X\subseteq S$ such that $S=XS^1.$
Therefore, we have that $S\!\setminus\!X\subseteq S^2,$ so $S$ has at most $|X|$ indecomposable elements.
\end{proof}

\begin{remark}
The class of right noetherian semigroups is strictly contained in the class of weakly right noetherian semigroups.
Indeed, every group is weakly right noetherian, but there exist groups that are not right noetherian.
\end{remark}

It is clear that the property of being right noetherian is a finiteness condition for semigroups:

\begin{lemma}
Finite semigroups are right noetherian.
\end{lemma}

For commutative semigroups, one-sided congruences coincide with (two-sided) congruences.
Right noetherian commutative semigroups have the following characterisation, which is due to Budach \cite{Budach} (necessity) and R{\'e}dei \cite{Redei} (sufficiency).

\begin{thm}
\label{commutative}
Let $S$ be a commutative semigroup.  Then $S$ is right noetherian if and only if $S$ is finitely generated.
\end{thm}

For groups, being right noetherian is equivalent to every subgroup being finitely generated (or, equivalently, satisfying the ascending chain condition on subgroups).

\begin{prop}
\label{group}
A group $G$ is right noetherian if and only if every subgroup of $G$ is finitely generated.
\end{prop}

\begin{proof}
($\Rightarrow$) Let $H$ be a subgroup of $G$.
We have a right congruence $\rho$ on $G$ given by
$$a\,\rho\,b\iff ab^{-1}\in H.$$
Since $G$ is right noetherian, we have that $\rho$ is generated by some finite subset $U$.
We shall show that $H$ is generated by the finite set 
$$X=\{xy^{-1} : (x, y)\in U\}.$$
Let $h\in H$.  We have that $h\,\rho\,1$, so there exists an $X$-sequence
$$h=x_1g_1, y_1g_1=x_2g_2, \dots, y_kg_k=1.$$
Hence, we have
$$h=(x_1y_1^{-1})x_2g_2=\cdots=(x_1y_1^{-1})\cdots(x_{k-1}y_{k-1}^{-1})x_kg_k=(x_1y_1^{-1})\dots(x_ky_k^{-1})\in\langle X\rangle.$$
($\Leftarrow$) Let $\rho$ be a right congruence on $G$.
The set $$H=\{xy^{-1} : (x, y)\in\rho\}$$ is a subgroup of $G$, so it is generated by some finite set $X$.
We claim that $\rho$ is generated by the set 
$$U=\{(x, y) : xy^{-1}\in X\}.$$
Indeed, let $a\,\rho\,b$ with $a\neq b$.
We have that $ab^{-1}\in H$, and hence
$$ab^{-1}=(x_1y_1^{-1})\dots(x_ky_k^{-1})$$ 
for some $x_iy_i^{-1}\in X$.  Therefore, we have a $U$-sequence
$$a=x_1(y_1^{-1}x_2\dots y_k^{-1}b), y_1(y_1^{-1}x_2\dots y_k^{-1}b)=x_2(y_2^{-1}x_3\dots y_k^{-1}b), \dots, y_k(y_k^{-1}b)=b,$$
as required.
\end{proof}

\begin{corollary}
\label{groupcorollary}
Let $G$ be a group and let $H$ be a subgroup of $G$.  If $G$ is right noetherian, then $H$ is right noetherian.
\end{corollary}

\begin{remark}
Given Proposition \ref{group}, from now on we will refer to right noetherian groups as {\em noetherian} groups.
It is known that polycyclic groups (and hence supersolvable and finitely generated nilpotent groups) are noetherian \cite{Hirsch}.
Also, there exist noetherian groups that are not finitely presented;
for example, Tarski groups 
(infinite groups $G$ all of whose non-trivial proper subgroups are cyclic of fixed prime order $p$),
which Olshanskii showed exist in \cite[Chapter 9]{Olshanskii}. 
\end{remark}

The following result shows that the property of being right noetherian is closed under homomorphic images (or equivalently quotients).

\begin{lemma}
\label{quotient}
Let $S$ be a semigroup and let $T$ be a homomorphic image of $S.$ 
If $S$ is right noetherian, then $T$ is right noetherian.
\end{lemma}

\begin{proof}
Let $\rho$ be a right congruence on $T$.
Let $\theta : S\to T$ be a surjective homomorphism, and define a right congruence $\rho^{\prime}$ on $S$ by 
$$s\,\rho^{\prime}\,t\iff s\theta\,\rho\,t\theta.$$
Since $S$ is right noetherian, $\rho^{\prime}$ is generated by a finite set $X.$
We claim that $\rho$ is generated by the finite set 
$$Y=\{(x\theta, y\theta) : (x, y)\in X\}.$$
Indeed, let $a\,\rho\,b$ with $a\neq b.$
There exist $s, t\in S$ such that $a=s\theta$ and $b=t\theta.$
Since $s\,\rho^{\prime}\,t$ and $s\neq t,$ there exists an $X$-sequence connecting $s$ and $t.$
Applying $\theta$ to every term of this sequence yields a $Y$-sequence connecting $a$ and $b.$
\end{proof}

It is often the case in semigroup theory that for a semigroup $S$ and an ideal $I$ of $S,$
if both $I$ and the Rees quotient $S/I$ have a property $\mathcal{P},$ then $S$ also has property $\mathcal{P}.$
In particular, this holds for the properties of being finitely generated and finitely presented \cite[Theorem 2.1]{Campbell}.
The next result reveals that this is also the case for the property of being right noetherian.

\begin{prop}
\label{idealextension}
Let $S$ be a semigroup and let $I$ be an ideal of $S$.
If both $I$ and $S/I$ are right noetherian, then $S$ is right noetherian.
\end{prop}

\begin{proof}
Let $\rho$ be a right congruence on $S.$
We have that $\rho_I$, the restriction of $\rho$ to $I,$ is generated by a finite set $X,$ since $I$ is right noetherian.
Let $U=S\!\setminus\!I$.  For each $s\in U$ such that $s\,\rho\,i$ for some $i\in I$, choose $\alpha(s)\in I$ such that $s\,\rho\,\alpha(s).$
Let $\rho^{\prime}$ be the right congruence on $S/I$ generated by the set 
$$Y=\bigl(\rho\cap(U\times U)\bigr)\cup\{(s, 0) : s\in U, s\,\rho\,i\text{ for some }i\in I\}.$$
Since $S/I$ is right noetherian, we have that $\rho^{\prime}$ is generated by a finite set $Y^{\prime}\subseteq Y$ by Lemma \ref{finite subset}.
Let $Y_1=Y^{\prime}\cap(U\times U)$ and let 
$$Y_2=\{(x, \alpha(x)) : (x, 0)\in Y^{\prime}\}.$$
We now make the following claim.
\begin{claim*}
Let $s\in S$ such that $s\,\rho\,i$ for some $i\in I$.
Then there exists $s^{\prime}\in I$ such that $(s, s^{\prime})$ is a consequence of $Y_1\cup Y_2.$
\end{claim*}
\begin{proof}
If $s\in I$, simply let $s^{\prime}=s$.  Suppose now that $s\in U$.
We have that $s\,\rho^{\prime}\,0$, so there exists a $Y^{\prime}$-sequence
$$s=x_1s_1, y_1s_1=x_2s_2, \dots, y_ks_k=0,$$
where $x_is_i\in U$ for $1\leq i\leq k.$
If $y_k\in U$, then we have a $Y_1$-sequence connecting $s$ and $y_ks_k\in I.$
If $y_k=0$, then $(s, x_ks_k)$ is a consequence of $Y_1,$ and we obtain $\alpha(x_k)s_k\in I$ from $x_ks_k$ using $Y_2.$
\end{proof}
Returning to the proof of Proposition \ref{idealextension}, we claim that $\rho$ is generated by the finite set $X\cup Y_1\cup Y_2.$\par
Indeed, let $s\,\rho\,t$ with $s\neq t.$
If there do not exist any $i\in I$ such that $(s, i)\in\rho$, then $s, t\in U$ and $(s, t)$ is a consequence of $Y_1$.
Otherwise, by the above claim, there exist $s^{\prime}, t^{\prime}\in I$ such that $(s, s^{\prime})$ and $(t, t^{\prime})$ are consequences of $Y_1\cup Y_2.$
Since $s^{\prime}\,\rho_I\,t^{\prime},$ we have that $(s^{\prime}, t^{\prime})$ is a consequence of $X.$
Hence, $(s, t)$ is a consequence of $X\cup Y_1\cup Y_2.$
\end{proof}

It is well known that for modules over Noetherian rings (rings that satisfy the ascending chain condition on right ideals), 
finite generation and finite presentability are equivalent properties.
Monoid acts play the analogous role in the theory of monoids as modules do in the theory of rings.
In the following, we provide some basic definitions about monoid acts; see \cite{Kilp} for more information.\par
Let $S$ be a semigroup.  A {\em (right) $S$-act} is a non-empty set $A$ together with a map 
$$A\times S\to A, (a, s) \mapsto as$$
such that $a(st)=(as)t$ for all $a\in A$ and $s, t\in S.$\par 
If $S$ is a monoid and $a1=a$ for all $a\in A,$ then $A$ is a {\em (monoid) $S$-act}.
For instance, $S$ itself is an $S$-act via right multiplication.\par 
An equivalence relation $\rho$ on an $S$-act $A$ is an {\em ($S$-act) congruence} on $A$ if $(a, b)\in\rho$ implies $(as, bs)\in\rho$ for all $a, b\in A$ and $s\in S.$
Note that the congruences on the $S$-act $S$ are precisely the right congruences on $S.$\par
An $S$-act $A$ is {\em finitely generated} if there exists a finite subset $X\subseteq A$ such that $A=XS^1,$
and $A$ is {\em finitely presented} if it is isomorphic to a quotient of a finitely generated free $S$-act by a finitely generated congruence.
One may consult \cite[Section 1.5]{Kilp} for more details.\par
Finite presentability of monoid acts has been systematically studied in \cite{Miller1}, \cite{Miller2}.
The following result provides a characterisation of right noetherian monoids in terms of finite presentability of their acts.

\begin{prop}
\label{acts}
The following are equivalent for a monoid $S$:
\begin{enumerate}
 \item $S$ is right noetherian;
 \item every cyclic $S$-act is finitely presented;
 \item every finitely generated $S$-act is finitely presented.
\end{enumerate}
\end{prop}

\begin{proof}
$(1)\Rightarrow(2)$.  Let $A$ be a cyclic $S$-act.
We have that $A\cong S/\rho$ where $\rho$ is a right congruence on $S$ \cite[Proposition 1.5.17]{Kilp}.  
Since $S$ is right noetherian, the right congruence $\rho$ is finitely generated, so $A$ is finitely presented.\par
$(2)\Rightarrow(3)$.  We prove by induction on the number of generators.  
Let $A=\langle X\rangle$ where $|X|=n.$  If $n=1$, then $A$ is cyclic and hence finitely presented, so assume that $n>1.$  
Assume that every finitely generated $S$-act with a generating set of size less than $n$ is finitely presented.
Choose $x\in X$ and let $B=\langle x\rangle$.  Since $B$ is cyclic, it is finitely presented.
If $A\!\setminus\!B$ is a subact, then it is generated by the set $X\!\setminus\!\{x\}$ and is hence finitely presented by assumption.
By \cite[Corollary 5.9]{Miller2} we have that $A,$ being the disjoint of two finitely presented $S$-acts, is also finitely presented.
If $A\!\setminus\!B$ is not a subact, we consider the Rees quotient $A/B=(A\!\setminus\!B)\cup\{0\},$ which is formed by defining $0\cdot s=0$ and
$$a\cdot s=
  \begin{cases} 
   as & \text{if }as\in A\\
   0 & \text{if }as\in B
  \end{cases}$$
for all $a\in A\!\setminus\!B$ and $s\in S.$ 
Now $A/B$ is generated by the set $X\!\setminus\!\{x\},$ so it is finitely presented by assumption.
It now follows from \cite[Corollary 4.4]{Miller2} that $A$ is finitely presented.\par
$(3)\Rightarrow(1)$.  For any right congruence $\rho$ on $S,$ we have that the cyclic $S$-act $A=S/\rho$ is finitely presented, 
and hence $\rho$ is finitely generated by \cite[Proposition 3.9]{Miller2}.
\end{proof}

\section{\large{Subsemigroups}\nopunct}

In this section we study the relationship between a semigroup and its substructures with regard to being right noetherian.
The following example demonstrates that subsemigroups (indeed, ideals) of right noetherian semigroups are not in general right noetherian.

\begin{ex}
\label{freecommex}
Let $S$ be the free commutative semigroup on $\{a, b\},$ which is right noetherian by Theorem \ref{commutative}, and let $I$ be the ideal of $S$ generated by $\{a\}$.
The set $\{ab^i : i\geq 0\}$ is a minimal (semigroup) generating set for $I,$ so $I$ is not right noetherian by Theorem \ref{commutative}.
\end{ex}

In the following we explore various situations in which the property of being right noetherian \textit{is} inherited by subsemigroups.\par

The first situation we consider is where a subsemigroup has finite complement.
It has been shown that various finiteness properties, including finite generation, finite presentability and periodicity, 
are inherited when passing to subsemigroups of finite complement and extensions; 
see \cite{Ruskuc1} for more details.
The following result shows that this also holds for the property of being right noetherian.

\begin{thm}
\label{large}
Let $S$ be a semigroup with a subsemigroup $T$ such that $S\!\setminus\!T$ is finite.
Then $S$ is right noetherian if and only if $T$ is right noetherian.
\end{thm}

\begin{proof}
($\Rightarrow$) Let $\rho$ be a right congruence on $T$.
Denote by $\overline{\rho}$ the right congruence on $S$ generated by $\rho$.
Since $S$ is right noetherian, $\overline{\rho}$ is generated by some finite subset $X\subseteq\rho.$
Define a set $$U=\{xs : (x, y)\in\overline{X}\text{ for some }y\in T, s\in S\!\setminus\!T\}.$$
Note that $U$ is finite, since $X$ and $S\!\setminus\!T$ are finite.
We claim that $\rho$ is generated by the finite set 
$$Y=X\cup\bigl(\rho\cap(U\times U)\bigr).$$
Indeed, let $a\,\rho\,b$ with $a\neq b.$
Since $a\,\overline{\rho}\,b$, there exists an $X$-sequence
\begin{equation}
a=x_1s_1, y_1s_1=x_2s_2, \dots, y_ks_k=b,
\end{equation}
where $(x_i, y_i)\in\overline{X}$ and $s_i\in S^1$ for $1\leq i\leq k.$
If every $s_i\in T^1,$ then $(a, b)$ is a consequence of $X.$
Otherwise, let $i$ be minimal such that $s_i\in S\!\setminus\!T,$ and let $j$ be maximal such that $s_j\in S\!\setminus\!T.$
We then have that $s_l\in T^1$ for every $l\in\{1, \dots, i-1, j+1, \dots, k\},$ and $(x_is_i, y_js_j)\in\rho\cap(U\times U).$ 
Therefore, by omitting the subsequence 
$$y_is_i=x_{i+1}s_{i+1}, \dots, y_{j-1}s_{j-1}=x_js_j$$
from (1), we obtain a $Y$-sequence connecting $a$ and $b,$ as required.\par
($\Leftarrow$) Let $\rho$ be a right congruence on $S.$
Let $\rho_T$ be the restriction of $\rho$ to $T.$
Since $T$ is right noetherian, we have that $\rho_T$ is generated by a finite set $X.$\par
Let $V=S\!\setminus\!T.$
For each $s\in V$ such that $s\,\rho\,t$ for some $t\in T$, choose $\alpha(s)\in T$ such that $s\,\rho\,\alpha(s)$.
We claim that $\rho$ is generated by the finite set 
$$Y=X\cup\{(s, \alpha(s)) : s\in V, s\,\rho\,t\text{ for some }t\in T\}\cup\bigl(\rho\cap(V\times V)\bigr).$$
Indeed, let $s\,\rho\,t$ with $s\neq t$.
If $s, t\in V$, then $(s, t)\in\rho\cap(V\times V)$.
If $s, t\in T$, then $s\,\rho_T\,t$, so $(s, t)$ is a consequence of $X$.
Finally, suppose that $s\in V$ and $t\in T$.
We have that $\alpha(s)\,\rho\,s\,\rho~t$, so $\alpha(s)\,\rho_T\,t$, and hence $(\alpha(s), t)$ is a consequence of $X$.
Therefore, we have that $(s, t)$ is a consequence of $Y$.
\end{proof}

\begin{corollary}
\label{adjoinidentity}
A semigroup $S$ is right noetherian if and only if $S^1$ is right noetherian.
\end{corollary}

\begin{corollary}
\label{adjoinzero}
A semigroup $S$ is right noetherian if and only if $S^0$ is right noetherian.
\end{corollary}

The next situation we consider is where the complement of a subsemigroup is a left ideal.

\begin{prop}
\label{leftideal}
Let $S$ be a semigroup with a subsemigroup $T$ such that $S\!\setminus\!T$ is a left ideal of $S.$
If $S$ is right noetherian, then $T$ is right noetherian.
\end{prop}

\begin{proof}
Let $\rho$ be a right congruence on $T.$
Let $\overline{\rho}$ denote the right congruence on $S$ generated by $\rho.$
Since $S$ is right noetherian, $\overline{\rho}$ is finitely generated, so it is generated by some finite subset $X\subseteq\rho.$
We claim that $\rho$ is generated by $X.$\par
Indeed, let $a\,\rho\,b$ with $a\neq b.$
Since $a\,\overline{\rho}\,b,$ there exists an $X$-sequence
$$a=x_1s_1, y_1s_1=x_2s_2, \dots, y_ks_k=b.$$
Since $a\in T$ and $S\!\setminus\!T$ is a left ideal of $S$, we must have that $s_1\in T^1.$
Therefore, we have that $y_1s_1\in T$, which in turn implies that $s_2\in T^1.$
Continuing in this way, we have that $s_i\in T^1$ for $1\leq i\leq k,$
so $(a, b)$ is a consequence of $X.$
\end{proof}

\begin{corollary}
\label{rnideal}
Let $S$ be a semigroup with a subsemigroup $T$ such that $S\!\setminus\!T$ is a right noetherian ideal of $S.$
Then $S$ is right noetherian if and only if $T$ is right noetherian.
\end{corollary}

\begin{proof}
The direct implication follows from Proposition \ref{leftideal}.
For the converse, let $I=S\!\setminus\!T$.  
Now $S/I\cong T\cup\{0\}$ is right noetherian by Corollary \ref{adjoinzero}, since $T$ is right noetherian.
It now follows from Proposition \ref{idealextension} that $S$ is right noetherian.
\end{proof}

\begin{remark}
If $S$ is a right noetherian semigroup with a right noetherian subsemigroup $T$ such that $I=S\!\setminus\!T$ is an ideal of $S,$
it is not necessary that $I$ be right noetherian.
Indeed, in Example \ref{freecommex} we have that $S$ and $T=S\!\setminus\!I\cong\mathbb{N}$ are both right noetherian, but $I$ is not right noetherian.
\end{remark}

Let $S$ be a semigroup and let $X$ be a subset of $S.$
We define the {\em right stabiliser} of $X$ in $S$ to be the set
$$\text{Stab}_R(X)=\{s\in S^1 : Xs=X\}.$$ 
It is clear that $\text{Stab}_R(X)$ is a submonoid of $S^1.$\par
Our next result states that the right stabiliser of any subset of a right noetherian semigroup is right noetherian.

\begin{prop}
\label{stabiliser}
Let $S$ be a semigroup and let $X$ be a subset of $S.$
If $S$ is right noetherian, then the right stabiliser $\emph{Stab}_R(X)$ of $X$ is also right noetherian.
\end{prop}

\begin{proof}
Let $\rho$ be a right congruence on $\text{Stab}_R(X)$.
Let $\overline{\rho}$ denote the right congruence on $S^1$ generated by $\rho$.
We have that $S^1$ is right noetherian by Corollary \ref{adjoinidentity}, so $\overline{\rho}$ is generated by some finite subset $Y\subseteq\rho.$
We claim that $\rho$ is generated by $Y.$\par
Indeed, let $a\,\rho\,b$ with $a\neq b$.  Since $a\,\overline{\rho}\,b$, there exists a $Y$-sequence
$$a=x_1s_1, y_1s_1=x_2s_2, \dots, y_ks_k=b,$$
where $(x_i, y_i)\in\overline{Y}$ and $s_i\in S^1$ for $1\leq i\leq k$.
Since $a, x_1\in\text{Stab}_R(X)$, we have $$Xs_1=(Xx_1)s_1=Xa=X,$$ so $s_1\in\text{Stab}_R(X)$.
Therefore, we have that $y_1s_1\in\text{Stab}_R(X)$, which in turn implies that $s_2\in\text{Stab}_R(X)$.
Continuing in this way, we have that $s_i\in\text{Stab}_R(X)$ for all $i\in\{1, \dots, k\},$ so $(a, b)$ is a consequence of $Y,$ as required.
\end{proof}

Of particular interest in semigroup theory are maximal subgroups, which coincide with the group $\mathcal{H}$-classes. 
Sch{\"u}tzenberger showed in \cite{Schutzenberger} how one can assign a group to an arbitrary $\mathcal{H}$-class,
so as to reflect the group-like properties of that class.
Here we give his construction; one may consult \cite{Lallement} for more details and basic properties of Sch{\"u}tzenberger groups.\par
Let $S$ be a semigroup and let $H$ be an $\mathcal{H}$-class of $S.$
Define a relation $\sigma_R(H)$ on the right stabiliser $\text{Stab}_R(H)$ by
$$(s, t)\in\sigma_R(H)\iff hs=ht\text{ for all }h\in H.$$
It is easy to see that $\sigma_R(H)$ is a congruence on $\text{Stab}_R(H),$ 
and it turns out that the quotient $\Gamma(H)=\text{Stab}_R(H)/\sigma_R(H)$ is a group.
Of course, by left-right duality, one may also define a group $\Gamma^{\prime}(H)$ by considering the left stabiliser of $H.$
It turns out, however, that $\Gamma(H)\cong\Gamma^{\prime}(H).$
We will therefore refer to $\Gamma(H)$ as the {\em Sch{\"u}tzenberger group} of $H$ and remember that is has two natural actions on $H,$ one on the left and one on the right.\par
Note that if $H$ is a group $\mathcal{H}$-class (that is, a maximal subgroup), then it is isomorphic to $\Gamma(H)$ \cite[Theorem 3.3]{Lallement}.
By \cite[Lemma 1.4]{Kozhukhov1}, maximal subgroups of right noetherian semigroups are noetherian.
Since $\Gamma(H)$ is a quotient of $\text{Stab}_R(H),$ 
Proposition \ref{stabiliser} and Lemma \ref{quotient} together yield the following generalisation of this result.

\begin{corollary}
\label{Schutzenberger}
Let $S$ be a semigroup and let $H$ be an $\mathcal{H}$-class of $S.$
If $S$ is right noetherian, then the Sch{\"u}tzenberger group $\Gamma(H)$ is noetherian.
\end{corollary}

We may now easily deduce that subgroups of right noetherian semigroups are noetherian.

\begin{corollary}\cite[Lemma 1.4]{Kozhukhov1}
\label{subgroup}
Let $S$ be a semigroup and let $G$ be a subgroup of $S$.  If $S$ is right noetherian, then $G$ is noetherian.
\end{corollary}

\begin{proof}
Let $e$ be the identity of $G$.  Now $G$ is a subgroup of the maximal subgroup $H_e\cong\Gamma(H_e)$.
We have that $H_e$ is noetherian by Corollary \ref{Schutzenberger}, and hence $G$ is noetherian by Corollary \ref{groupcorollary}.
\end{proof}

Kozhukhov proved in \cite{Kozhukhov1} that, given a semigroup $S$ and a subgroup $G$ of $S,$ the lattice of subgroups of $G$ can be embedded in the lattice of right congruences on $S.$
Our next result shows that Sch{\"u}tzenberger groups exhibit the same behaviour.

\begin{prop}
\label{prop:Schutz}
Let $S$ be a semigroup and let $H$ be an $\mathcal{H}$-class of $S.$ 
Then the lattice $L\bigl(\Gamma(H)\bigr)$ of subgroups of $\Gamma(H)$ can be embedded in the lattice $L(S)$ of right congruences on $S.$
\end{prop}

\begin{proof}
We consider the left action of $\Gamma(H)$ on $H,$ and denote $\sigma_L(H)$ by $\sigma.$
We define an action of $\Gamma(H)$ on $HS^1$ as follows: for $s/\sigma\in\Gamma(H)$ and $x\in HS^1,$ let $s/\sigma\!\cdot\!x=sx.$  We claim that this action is well-defined.
Indeed, if $(s, t)\in\sigma$ and $x\in HS^1,$ then $x=hu$ for some $h\in H$ and $u\in S^1,$ and hence $$sx=(sh)u=(th)u=tx.$$
We say that $\Gamma(H)$ acts {\em faithfully} on $x\in HS^1$ if $g_1\!\cdot\!x=g_2\!\cdot\!x$ implies that $g_1=g_2$ for all $g_1, g_2\in\Gamma(H).$
Note that for any $h, h^{\prime}\in H$ and $s\in S^1,$ we have that $\Gamma(H)$ acts faithfully on $hs$ if and only if it acts faithfuly on $h^{\prime}s.$\par
Now let $G$ be a subgroup of $\Gamma(H).$ 
We define a relation $\rho_G$ on $S$ as follows: $(x, y)\in\rho_G$ if and only if one of the following holds:
\begin{enumerate}
 \item $x=y$;
 \item $x=hs, y=h^{\prime}s$ for some $h, h^{\prime}\in H, s\in S^1,$ $h=g\!\cdot\!h^{\prime}$ for some $g\in G,$ and $\Gamma(H)$ acts faithfully on $x.$
 \item $x=hs, y=h^{\prime}s$ for some $h, h^{\prime}\in H, s\in S,$ and $\Gamma(H)$ does not act faithfully on $x.$
\end{enumerate}
It can be easily shown that $\rho_G$ is a right congruence.
We claim that the map $G\mapsto\rho_G$ is a lattice embedding of $L\bigl(\Gamma(H)\bigr)$ into $L(S).$
So, let $G_1$ and $G_2$ be subgroups of $\Gamma(H).$
We need to prove the following: (a) $\rho_{G_1\cap G_2}=\rho_{G_1}\cap\rho_{G_2}$; (b) $\rho_U=\langle\rho_{G_1}\cup\rho_{G_2}\rangle$, where $U$ is the subgroup of $G$ generated by $G_1\cup G_2$;
(c) $\rho_{G_1}=\rho_{G_2}$ implies that $G_1=G_2$.\par 
(a) It is clear that $\rho_{G_1\cap G_2}\subseteq\rho_{G_1}\cap\rho_{G_2},$ so we just need to prove the reverse containment.
Let $(x, y)\in\rho_{G_1}\cap\rho_{G_2}.$  
If (1) or (3) are satisfied, then $(x, y)\in\,\rho_{G_1\cap G_2}$ (since (1) and (3) do not refer to $G$), so suppose that (2) holds.
Therefore, there exist $h_1, h_1^{\prime}\in H,$ $s_1\in S^1$ and $g_1\in G_1$ such that $x=h_1s_1, y=h_1^{\prime}s_1, h_1=g_1\!\cdot\!h_1^{\prime},$
and there exist $h_2, h_2^{\prime}\in H,$ $s_2\in S^1$ and $g_2\in G_2$ such that $x=h_2s_2, y=h_2^{\prime}s_2, h_2=g_2\!\cdot\!h_2^{\prime}.$ 
It follows that $g_1\!\cdot\!y=x=g_2\!\cdot\!y,$ which implies that $g_1=g_2,$ and hence $(x, y)\in\rho_{G_1\cap G_2}.$\par 
(b) Again, it suffices to consider $(x, y)$ that satisfies (2).  Suppose first that $(x, y)\in\rho_U.$
Then $x=hs, y=h^{\prime}s$ for some $h, h^{\prime}\in H, s\in S^1,$ and $h=u\!\cdot\!h^{\prime}$ for some $u\in U.$
Now $u=g_1\dots g_n$ for some $g_i\in G_1\cup G_2.$  
Letting $h_i=(u_i\dots u_n)\!\cdot\!h^{\prime}, h_{n+1}=h^{\prime}$ and $x_i=h_is,$ 
we have that $h_i=u_i\!\cdot\!h_{i+1}$ and $(x_i, x_{i+1})\in\rho_{G_1}\cup\rho_{G_2}$ for each $i\in\{1, \dots, n\}.$
Hence, we have $(x, y)=(x_1, x_{n+1})\in\langle\rho_{G_1}\cup\rho_{G_2}\rangle.$\par 
Now suppose that $(x, y)\in\langle\rho_{G_1}\cup\rho_{G_2}\rangle.$  Then there exists a sequence $$x=x_1t_1, y_1t_1=x_2t_2, \dots, y_kt_k=y,$$
where, for each $i\in\{1, \dots, k\},$ there exist $h_i, h_i^{\prime}\in H,$ $s_i\in S^1$ and $g_i\in G_1\cup G_2$ such that $x_i=h_is_i, y_i=h_i^{\prime}s_i, h_i=g_i\!\cdot\!h_i^{\prime}.$
Letting $u=g_1\dots g_k\in U$ and $h=u\!\cdot\!h_k^{\prime}\in H,$ it follows from the above sequence that $x=u\!\cdot\!y=h(s_kt_k),$ and hence $(x, y)\in\rho_U.$\par 
(c) Suppose that $\rho_{G_1}=\rho_{G_2}.$  Consider $g_1\in G_1,$ let $x\in H$ be arbitrary, and let $y=g_1\!\cdot\!x\in H.$
We have that $(x, y)\in\rho_{G_1}=\rho_{G_2}$ via (2), so there exist $h, h^{\prime}\in H,$ $s\in S^1$ and $g_2\in G_2$ such that $x=hs, y=h^{\prime}s, h^{\prime}=g_2\!\cdot\!h.$
Hence, we have that $$g_1\!\cdot\!x=y=h^{\prime}s=g_2\!\cdot\!hs=g_2\!\cdot\!x,$$ which implies that $g_1=g_2\in G_2,$ so $G_1\subseteq G_2.$
A similar argument proves that $G_2\subseteq G_1,$ and hence $G_1=G_2,$ as required.
\end{proof}

A semigroup $S$ is said to be {\em regular} if for every $s\in S$ there exists $x\in S$ such that $s=sxs.$
If, additionally, the inverse of each element of $S$ is unique, then $S$ is an {\em inverse semigroup}.
It is well known that a semigroup is inverse if and only if it is regular and its idempotents commute \cite[Proposition 2.12]{Lallement}.\par
We have the following characterisation, due to Kozhukhov, of right noetherian inverse semigroups.

\begin{thm}\cite[Theorem 4.3]{Kozhukhov1}
\label{inverse}
Let $S$ be an inverse semigroup.  Then the following are equivalent:
\begin{enumerate}
 \item $S$ is right noetherian;
 \item $S$ has finitely many idempotents and all its maximal subgroups are noetherian.
\end{enumerate}
\end{thm}

The above result leads to the following alternative characterisation of right noetherian inverse semigroups,
which is a generalisation of Proposition \ref{group}.

\begin{corollary}
\label{inversesubsemigroupsfg}
Let $S$ be an inverse semigroup.
Then $S$ is right noetherian if and only if every inverse subsemigroup of $S$ is finitely generated.
\end{corollary}

\begin{proof}
($\Rightarrow$) Let $T$ be an inverse subsemigroup of $S$.
It follows from Theorem \ref{inverse} that $T$ has finitely many idempotents and all its maximal subgroups are finitely generated,
so $T$ is finitely generated by \cite[Proposition 3.1]{Ruskuc2}.\par
($\Leftarrow$) We have that the semilattice $Y$ of idempotents of $S$ is finite (since it is finitely generated).
Also, since all the subgroups of $S$ are finitely generated, we have that all the maximal subgroups of $S$ are noetherian.
Hence, by Theorem \ref{inverse}, we have that $S$ is right noetherian.
\end{proof}

The next result shows that the property of being right noetherian is inherited by inverse subsemigroups.

\begin{prop}
\label{inversesubsemigroup}
Let $S$ be a semigroup and let $T$ be an inverse subsemigroup of $S.$
If $S$ is right noetherian, then $T$ is right noetherian.
\end{prop}

\begin{proof}
We show that the semilattice $Y$ of idempotents of $T$ is right noetherian.  
It then follows from Theorem \ref{commutative} that $Y$ is finitely generated and hence finite.
Since, by Corollary \ref{subgroup}, the maximal subgroups of $T$ are noetherian, 
we then have that $T$ is right noetherian by Theorem \ref{inverse}.\par 
So, let $\rho$ be a (right) congruence on $Y.$
Let $\overline{\rho}$ denote the right congruence on $S$ generated by $\rho.$
Since $S$ is right noetherian, $\overline{\rho}$ is generated by some finite subset $X\subseteq\rho.$
Let $U$ be the subsemigroup of $Y$ generated by the set
$$P=\{e\in Y : (e, f)\in\overline{X}\text{ for some }f\in Y\}.$$
Notice that $U$ is finite since $P$ is finite.
We claim that $\rho$ is generated by the set $Z=\rho\cap(U\times U).$ (Note that $X\subseteq Z.$)\par 
Indeed, let $e\,\rho\,f$ with $e\neq f.$
Since $e\,\overline{\rho}\,f,$ there exists an $X$-sequence
$$e=e_1s_1, f_1s_1=e_2s_2, \dots, f_ks_k=f,$$
where $(e_i, f_i)\in\overline{X}$ and $s_i\in S^1$ for $1\leq i\leq k.$\par 
We claim that $e_1\dots e_is_ie\in Y$ for every $i\in\{1, \dots, k\}.$
Indeed, we have that $e_1s_1e=e\in Y$ and, for $i\in\{2, \dots, k\},$ we have
$$e_1\dots e_is_ie=e_1\dots e_{i-1}f_{i-1}s_{i-1}e=f_{i-1}(e_1\dots e_{i-1}s_{i-1}e),$$
so the claim follows by induction.  We now have an $X$-sequence
$$e=e_1(e_1s_1e), f_1(e_1s_1e)=e_2(e_1e_2s_2e), \dots, f_k(e_1\dots e_ks_ke)=(e_1\dots e_k)(ef).$$
By a symmetrical argument, we also have that $\bigl(f, (f_1\dots f_k)(ef)\bigr)$ is a consequence of $X.$
Since $(e_1\dots e_k, f_1\dots f_k)\in Z,$ it follows that $(e, f)$ is a consequence of $Z.$
\end{proof}

\begin{corollary}
\label{semilattice}
Right noetherian semigroups contain no infinite semilattices.
\end{corollary}

It is well known that if a monoid $M$ has a non-unit $x$ such that $xy=1$ for some $y\in M,$
then the elements $y^ix^i$ are distinct and form a semilattice.
Hence, we have:

\begin{corollary}
\label{unit}
Let $M$ be a right noetherian monoid.
Then every right (or left) unit of $M$ is in fact a (two-sided) unit.
Equivalently, the complement of the group of units of $M$ is an ideal.
\end{corollary}

\section{Direct products\nopunct\\}

In this section we investigate under what conditions the direct product $S\times T$ of two semigroups is right noetherian.
Since $S$ and $T$ are homomorphic images of $S\times T,$ by Lemma \ref{quotient} we have:

\begin{lemma}
\label{directfactors}
Let $S$ and $T$ be two semigroups.  If $S\times T$ is right noetherian, then both $S$ and $T$ are right noetherian.
\end{lemma}

The direct product of two finitely generated semigroups is not necessarily finitely generated.
Necessary and sufficient conditions for a direct product of two semigroups to be finitely generated were given in \cite{Robertson}.

\begin{thm}\cite[Theorem 2.1, Theorem 8.2]{Robertson}
\label{dpfg}
Let $S$ and $T$ be two semigroups with $S$ infinite.
\begin{enumerate}
 \item If $T$ is infinite, then $S\times T$ is finitely generated if and only if both $S$ and $T$ are finitely generated and $S^2=S$ and $T^2=T.$
 \item If $T$ is finite, then $S\times T$ is finitely generated if and only if $S$ is finitely generated and $T^2=T.$
\end{enumerate}
\end{thm}

This raises the question as to whether there might exist a direct product $S\times T$ of two finitely generated right noetherian semigroups such that $S\times T$ is right noetherian but not finitely generated.
However, this turns out to not be the case.

\begin{lemma}
\label{dp factor indecomposable}
Let $S$ and $T$ be two semigroups.
If $S$ is infinite and $S\times T$ is right noetherian, then $T^2=T.$ 
\end{lemma}

\begin{proof}
If $T^2\neq T,$ then $T$ has an indecomposable element $t.$
Therefore, we have that $(s, t)$ is indecomposable for every $s\in S,$ so $S\times T$ has infinitely many indecomposable elements.
It follows from Corollary \ref{indecomposable} that $S\times T$ is not right noetherian.
\end{proof}

\begin{corollary}
Let $S$ and $T$ be two finitely generated right noetherian semigroups.
If the direct product $S\times T$ is right noetherian, then it is finitely generated.
\end{corollary}

\begin{proof}
Assume that $S$ is infinite.  Then $T^2=T$ by Lemma \ref{dp factor indecomposable}.
If $T$ is infinite, then $S^2=S$ by Lemma \ref{dp factor indecomposable}, and hence $S\times T$ is finitely generated by Theorem \ref{dpfg}(1).
If $T$ is finite, then $S\times T$ is finitely generated by Theorem \ref{dpfg}(2).
\end{proof}

It can now be easily seen that the converse of Lemma \ref{directfactors} does not hold.
Indeed, for any infinite right noetherian semigroup $S,$
the direct product $\mathbb{N}\times S,$ where $\mathbb{N}$ is the free monogenic semigroup,
is not right noetherian since $\mathbb{N}$ contains an indecomposable element.
We now present a more striking example:

\begin{ex}
\label{dpex}
\textit{If $S$ is any infinite right noetherian semigroup and $T=\{a, b\}$ is the two-element right zero semigroup, then $S\times T$ is not right noetherian.}\par
Let $\rho$ be the right congruence on $S\times T$ generated by the set 
$$X=\bigl\{\bigl((s, a), (s, b)\bigr) : s\in S\bigr\}.$$
We claim that $X$ is a minimal generating set for $\rho.$
Indeed, if $s\in S$ and we write $(s, a)$ as a product $(s, a)=uv$ for some $u, v\in S\times T,$
we have that $u=(x, c)$ for some $x\in S$ and $c\in T,$ and $v=(y, a)$ for some $y\in S.$
Now, applying the generating pair $\bigl((x, c), (x, d)\bigr),$ 
where $c\neq d\in T,$ to $(s, a),$ we obtain $(x, d)(y, a)=(s, a).$
Therefore, we cannot have that $\bigl((s, a), (s, b)\bigr)$ is a consequence of $X\!\setminus\!\bigl\{\bigl((s, a), (s, b)\bigr)\bigr\}.$
\end{ex}

\begin{remark}
\label{regularremark}
If the semigroup $S$ in Example \ref{dpex} is a group, then $S\times T$ is a completely simple semigroup with two maximal subgroups that are both noetherian.
Hence, Theorem \ref{inverse} does not generalise to regular semigroups.
\end{remark}

The direct product of a right noetherian semigroup with a right noetherian monoid may not be right noetherian;
in fact, Example \ref{dpex} shows that there exist finite semigroups $S$ such that the direct product of $S$ with any infinite monoid $M$ is not right noetherian.
This leads us to make the following definition.

\begin{defn}
Let $S$ be a right noetherian semigroup.
We say that $S$ {\em preserves right noetherian in direct products} if is satisfies the following:
For every monoid $M,$ the direct product $S\times M$ is right noetherian if and only if $M$ is right noetherian.
\end{defn}

We will show in what follows that a large class of right noetherian semigroups have the above property.
We have seen that it easy to find right noetherian semigroups that do not have the property, 
but it turns out that this is not the case for monoids, and so we raise the following question.

\begin{prob}
Do there exist right noetherian monoids $M$ and $N$ such that the direct product $M\times N$ is not right noetherian?
\end{prob}

We now show how, given a right congruence on the direct product $S\times M$ of a semigroup $S$ and a monoid $M,$ 
one can naturally construct a family of right congruences on $S.$
So, let $\rho$ be a right congruence on $S\times M.$
For each $m\in M,$ we define a right congruence $\rho_m^S$ on $S$ by
$$s\,\rho_m^S\,t\iff(s, m)\,\rho\,(t, m).$$
To see that $\rho_m^S$ is a right congruence, let $s\,\rho_m^S\,t$ and $u\in S.$
Since $\rho$ is a right congruence on $S,$ we have that
$$(su, m)=(s, m)(u, 1)\,\rho\,(t, m)(u, 1)=(tu, m),$$
so $su\,\rho_m^S\,tu.$\par
When there is no danger of confusion, we will usually just write $\rho_m$ for $\rho_m^S.$\par
The following result shows that the property of preserving right noetherian in direct products is a finiteness condition for monoids.

\begin{thm}
\label{dpfinitemonoid}
Let $M$ be a finite monoid.  
Then, for any semigroup $S,$ the direct product $M\times S$ is right noetherian if and only if $S$ is right noetherian.
\end{thm}

\begin{proof}
Let $S$ be a right noetherian semigroup, and let $\rho$ be a right congruence on $M\times S$.
Since $S$ is right noetherian, each $\rho_m$ is generated by a finite set $X_m$.
Let $$Y_m=\bigl\{\bigl((m, x), (m, y)\bigr) : (x, y)\in X_m\bigr\}.$$
Now let $Q$ be the set 
$$\{(m, n)\in M\times M : m\neq n, (m, s)\,\rho\,(n, t)\text{ for some }s, t\in S\}.$$
For each pair $(m, n)\in Q$, we define a set 
$$I(m, n)=\{s\in S : (m, s)\,\rho\,(n, t)\text{ for some } t\in S\}.$$
We claim that $I(m, n)$ is a right ideal of $S.$
Indeed, let $s\in I(m, n)$ and $u\in S.$
We then have that $(m, s)\,\rho\,(n, t)$ for some $t\in S,$ and therefore
$$(m, su)=(m, s)(1, u)\,\rho\,(n, t)(1, u)=(n, tu),$$
so $su\in I(m, n).$\par
Since $S$ is weakly right noetherian, we have that $I(m, n)$ is generated (as a right ideal) by some finite set $P(m ,n).$ 
For each $p\in P(m, n),$ choose $\alpha_{m, n}(p)\in S$ such that
$$(m, p)\,\rho\,(n, \alpha_{m, n}(p)).$$
Let $H$ be the finite set 
$$\bigl\{\bigl((m, p),(n, \alpha_{m, n}(p))\bigr) : (m, n)\in Q, p\in P(m, n)\bigr\}.$$
We claim that $\rho$ is generated by the finite set 
$$Y=\biggl(\bigcup_{m\in M}Y_m\biggr)\cup H.$$
Indeed, let $(m, s)\,\rho\,(n, t)$ with $(m, s)\neq(n, t).$
Suppose first that $m=n.$ 
Since $s\,\rho_m\,t$ with $s\neq t,$ there exists an $X_m$-sequence
$$s=x_1s_1, y_1s_1=x_2s_2, \dots, x_ks_k=t.$$
Therefore, we have a $Y_m$-sequence
$$(m, s)=(m, x_1)(1, s_1), (m, y_1)(1, s_1)=(m, x_2)(1, s_2), \dots, (m, y_k)(1, s_k)=(m, t).$$
Now suppose that $m\neq n.$
We claim that there exists $s^{\prime}\in S$ such that the pair $\bigl((m, s), (n, s^{\prime})\bigr)$ is a consequence of $H.$
Indeed, since $s\in I(m, n)$, we have that $s=pu$ for some $p\in P(m, n)$ and $u\in S^1.$
Let $s^{\prime}=\alpha_{m, n}(p)u.$
Either $u=1,$ in which case $s=p$ and $\bigl((m, s),(n, s^{\prime})\bigr)\in H,$
or we have an $H$-sequence 
$$(m, s)=(m, p)(1, u), (n, \alpha_{m, n}(p))(1, u)=(n, s^{\prime}).$$
It now follows that $(n, s^{\prime})\,\rho\,(n, t)$, so $s^{\prime}\,\rho_n\,t$.
Therefore, either $s^{\prime}=t$ or there exists a $Y_n$-sequence connecting $(n, s^{\prime})$ and $(n, t).$
Hence, $\bigl((m, s), (n, t)\bigr)$ is a consequence of $Y.$
\end{proof}

We now exhibit an example of a right noetherian semigroup with a finitely generated subsemigroup that is not right noetherian.

\begin{ex}
Let $S$ be any infinite finitely generated right noetherian semigroup,
and let $T$ be the right zero semigroup on $\{a, b\}$.
By Theorem \ref{dpfinitemonoid}, we have that $S\times T^1$ is right noetherian.
Now, we showed in Example \ref{dpex} that the subsemigroup $S\times T$ of $S\times T^1$ is not right noetherian. 
However, since $T^2=T$, the direct product $S\times T$ is finitely generated by Theorem \ref{dpfg}.
\end{ex}

We now prove some lemmas that will be useful in the remainder of the section.

\begin{lemma}
\label{dpquotient}
Let $S$ be a semigroup and let $T$ be a homomorphic image of $S.$
If $S$ preserves right noetherian in direct products, then so does $T.$
\end{lemma}

\begin{proof}
Let $M$ be a right noetherian monoid.
Since $S\times M$ is right noetherian and $T\times M$ is a homomorphic image of $S\times M,$
we have that $T\times M$ is right noetherian by Lemma \ref{quotient}.
\end{proof}

\begin{lemma}
\label{dpdp}
Let $S$ be a semigroup and let $M$ be a monoid.
If both $S$ and $M$ preserve right noetherian in direct products, then so does the direct product $S\times M.$
\end{lemma}

\begin{proof}
Let $N$ be a right noetherian monoid.
Since $M$ preserves right noetherian in direct products, we have that $M\times N$ is right noetherian.
Now, by the assumption that $S$ preserves right noetherian in direct products, 
we have that $(S\times M)\times N\cong S\times(M\times N)$ is right noetherian.
\end{proof}

\begin{lemma}
\label{dpidealextension}
Let $S$ be a semigroup and let $I$ be an ideal of $S.$
If both $I$ and the Rees quotient $S/I$ preserve right noetherian in direct products, then so does $S.$
\end{lemma}

\begin{proof}
Let $M$ be a right noetherian monoid, and let $J=I\times M.$  
We have that $J$ is an ideal of $S\times M,$ and $J$ is right noetherian since $I$ preserves right noetherian in direct products. 
We have a surjective homomorphism 
$$(S/I)\times M\to(S\times M)/J, ([s]_I, m)\mapsto [(s, m)]_J.$$
Since $(S/I)\times M$ is right noetherian, we have that $(S\times M)/J$ is right noetherian by Lemma \ref{quotient}.
It now follows from Proposition \ref{idealextension} that $S\times M$ is right noetherian.
\end{proof}

\begin{lemma}
\label{dpzero}
Let $S$ be a semigroup.  Then $S$ preserves right noetherian in direct products if and only if $S^0$ does.
\end{lemma}

\begin{proof}
Let $M$ be a right noetherian monoid. 
Since $\{0\}\times M$ is a right noetherian ideal of $S^0\times M,$
it follows from Corollary \ref{rnideal} that $S\times M$ is right noetherian if and only $S^0\times M$ is right noetherian.
\end{proof}

\begin{remark}
The analogue of Lemma \ref{dpzero} for $S^1$ does not hold.
Indeed, the two-element right zero semigroup $T$ does not preserve right noetherian in direct products but $T^1$ does.
\end{remark}

\begin{lemma}
\label{dpf}
Let $M$ and $N$ be two right noetherian monoids, and let $\rho$ be a right congruence on $M\times N.$
Then there exists a finite set $X\subseteq M$ with the following property:
for each $m\in M,$ there exist $x\in X,$ $m^{\prime}\in M$ such that $m=xm^{\prime}$ and $\rho_m=\rho_x.$
\end{lemma}

\begin{proof}
First note that for any $m, n\in M$, we have $\rho_m\subseteq\rho_{mn}$.
Suppose that there are infinitely many pairwise disjoint right congruences of the form $\rho_m$.
Let $I_1$ be the right ideal
$$\{m\in M : \rho_m\text{ is not the identity congruence on }N\}.$$
Since $M$ is weakly right noetherian, $I_1$ is generated by some finite set $Y_1$.
Now, there exists $y_1\in Y_1$ such that there are infinitely many pairwise disjoint right congruences of the form $\rho_m$ where $m$ is in the right ideal $\langle y_1\rangle.$
Let $I_2$ be the right ideal
$$\{m\in\langle y_1\rangle : \rho_m\neq\rho_{y_1}\}.$$
We have that $I_2$ is generated by some finite set $Y_2$, 
and there exists $y_2\in Y_2$ such that there are infinitely many pairwise disjoint right congruences of the form $\rho_m, m\in\langle y_2\rangle$.
Continuing in this way, we have an infinite chain of right congruences
$$\rho_{y_1}\subset\rho_{y_2}\subset\dots.$$
But $N$ is right noetherian, so we have a contradiction.
Hence, there exists a finite set $U$ such that for each $m\in M$, there exists $u\in U$ with $\rho_m=\rho_u$.\par
For each $u\in U$, let $I_u$ be the right ideal generated by the set $\{m\in M : \rho_m=\rho_u\}$.
Since $M$ is weakly right noetherian, we have that $I_u$ is generated by some finite set $Y_u$.
Now set $X=\bigcup_{u\in U}Y_u$.  It is clear that $X$ satisfies the condition in the statement of the lemma.
\end{proof}

We now show that finitely generated commutative monoids preserve right noetherian in direct products.

\begin{thm}
\label{dpfgcommutative}
Let $M$ be a finitely generated commutative monoid.  
Then, for any monoid $N,$ the direct product $M\times N$ is right noetherian if and only if $N$ is right noetherian.
\end{thm}

\begin{proof}
Since $M$ is a homomorphic image of the free commutative monoid $\mathbb{N}_0^n$ on $n$ generators for some $n$,
it follows from Lemma \ref{dpquotient} and Lemma \ref{dpdp} that it is sufficient to prove that $\mathbb{N}_0$ preserves right noetherian in direct products.\par
So, let $U$ be a right noetherian monoid, and let $\rho$ be a right congruence on $\mathbb{N}_0\times U.$
Let $\rho^{\prime}$ be the right congruence on $U$ generated by the set 
$$R=\{(u, v)\in U\times U : (a, u)\,\rho\,(b, v)\text{ for some }a, b\in\mathbb{N}_0\}.$$
Since $U$ is right noetherian, $\rho^{\prime}$ is generated by a finite set $X\subseteq R.$
For each pair $(x, y)\in\overline{X},$ choose $\alpha(x, y), \beta(x, y)\in\mathbb{N}_0$ such that
$$(\alpha(x, y), x)\,\rho\,(\beta(x, y), y),$$
and let $$Y=\bigl\{\bigl((\alpha(x, y), x), (\beta(x, y), y)\bigr) : (x, y)\in\overline{X}\bigr\}.$$
Let $N$ be the maximum natural number in the set 
$$\{\alpha(x, y), \beta(x, y) : (x, y)\in\overline{X}\}.$$
Now let $Q$ be the set 
$$\{(a, b)\in\mathbb{N}_0\times\mathbb{N}_0: a<b<N, (a, u)\,\rho\,(b, v)\text{ for some }u, v\in U\}.$$
For each pair $(a, b)\in Q$, we define a right ideal 
$$I(a, b)=\{u\in N : (a, u)\,\rho\,(b, v)\text{ for some }v\in U\}.$$
Since $U$ is weakly right noetherian, we have that $I(a, b)$ is generated by some finite set $P(a ,b)$. 
For each $p\in P(a, b)$, choose $\alpha_{a, b}(p)\in U$ such that
$$(a, p)\,\rho\,(b, \alpha_{a, b}(p)).$$
Let $H$ be the finite set 
$$\bigl\{\bigl((a, p),(b, \alpha_{a, b}(p))\bigr) : (a, b)\in Q, p\in P(a, b)\bigr\}.$$
Now, for each $a<N$, we have that $\rho_a^U$ is generated by some finite set $X_a^U$.
Let $$Y_a^U=\bigl\{\bigl((a, x), (a, y)\bigr) : (x, y)\in X_a^U\bigr\}.$$
By Lemma \ref{dpf}, there exists a finite set $S$ such that, for each $u\in U,$ 
there exist $s\in S$ and $u^{\prime}\in U$ such that $u=su^{\prime}$ and $\rho_u^{\mathbb{N}_0}=\rho_s^{\mathbb{N}_0}.$
For each $s\in S$, we have that $\rho_s^{\mathbb{N}_0}$ is generated by some finite set $X_s^{\mathbb{N}_0}$.
Let $$Y_s^{\mathbb{N}_0}=\bigl\{\bigl((x, s), (y, s)\bigr) : (x, y)\in X_s^{\mathbb{N}_0}\bigr\}.$$
We claim that $\rho$ is generated by the finite set 
$$Z=Y\cup H\cup\biggl(\bigcup_{a<N}Y_a^U\biggr)\cup\biggl(\bigcup_{s\in S}Y_s^{\mathbb{N}_0}\biggr).$$
Let $(a, u)\,\rho\,(b, v)$ with $(a, u)\neq(b, v).$
Suppose first that $u=v.$  Then $a\,\rho_u^{\mathbb{N}_0}\,b.$
Now, there exist $s\in S$ and $u^{\prime}\in U$ such that $u=su^{\prime}$ and $\rho_u^{\mathbb{N}_0}=\rho_s^{\mathbb{N}_0}.$
Therefore, we have that $(a, b)$ is a consequence of $X_s^{\mathbb{N}_0},$ 
and hence we have a $Y_s^{\mathbb{N}_0}$ connecting $(a, u)=(a, s)(0, u^{\prime})$ and $(b, v)=(b, u)=(b, s)(0, u^{\prime}).$\par
Now suppose that $u\neq v.$
We have that $u\,\rho^{\prime}\,v,$ so there exists an $X$-sequence
$$u=x_1u_1, y_1u_1=x_2u_2, \dots, y_ku_k=v.$$
Let $\alpha_i=\alpha(x_i, y_i)$ and $\beta_i=\beta(x_i, y_i)$ for $i\in\{1, \dots, k\}.$
Now let $a_1=a-\alpha_1$ and $a_i=\beta_{i-1}+a_{i-1}-\alpha_i$ for $j\in\{2, \dots, k\}.$
Suppose first that $a_i\geq 0$ for each $i\in\{1, \dots, k\}.$
We then have a $Y$-sequence
\begin{align*}
(a, u)=(\alpha_1, x_1)(a_1, u_1), (\beta_1, y_1)(a_1, u_1)=&(\alpha_2, x_2)(a_2, u_2),\\
\dots, &(\beta_k, y_k)(a_k, u_k)=(a^{\prime}, v),
\end{align*}
where $a^{\prime}=\beta_k+a_k.$  
Now, either $a^{\prime}=b$ or, by exactly the same argument as above, $\bigl((a^{\prime}, v),(b, v)\bigr)$ is a consequence of $Y_s^{\mathbb{N}_0}.$\par
Now suppose that there exists $i\in\{1, \dots, k\}$ such that $a_i<0,$
and assume that $i$ is minimal such that $a_i<0.$ 
Let $b_k=b-\beta_k$ and $b_i=\alpha_{i+1}+b_{i+1}-\beta_i$ for $i\in\{1, \dots, k-1\}.$
Using a similar argument as above, we have that either $\bigl((b, v), (a, u)\bigr)$ is a consequence of $Z,$
or there exists $j$ maximal such that $b_j<0.$
In the latter case, let 
$$u^{\prime}=x_iu_i,\, a^{\prime}=\beta_{i-1}+a_{i-1},\, v^{\prime}=y_ju_j,\, b^{\prime}=\alpha_{j+1}+b_{j+1}.$$
Note that $a^{\prime}, b^{\prime}<N.$  By a similar argument to the one in Theorem \ref{dpfinitemonoid}, 
we have that $\bigl((a^{\prime}, u^{\prime}),(b^{\prime}, v^{\prime})\bigr)$ is a consequence of 
$$H\cup\biggl(\bigcup_{a<N}Y_a^U\biggr).$$
Hence, $\bigl((a, u),(b, v)\bigr)$ is a consequence of $Z.$
\end{proof}

\begin{remark}
Theorem \ref{dpfgcommutative} does not generalise to finitely generated commutative semigroups.
Indeed, we have already observed that $\mathbb{N}\times M$ is not right noetherian for any infinite right noetherian monoid $M.$
\end{remark}

Our next result states that right noetherian inverse semigroups preserve right noetherian in direct products.

\begin{thm}
\label{dpinverse}
Let $S$ be a right noetherian inverse semigroup.
Then, for any monoid $M,$ the direct product $S\times M$ is right noetherian if and only if $M$ is right noetherian.
\end{thm}

In order to prove Theorem \ref{dpinverse}, we first prove a couple of lemmas.

\begin{lemma}
\label{dpgroup}
Any noetherian group $G$ preserves right noetherian in direct products.
\end{lemma}

\begin{proof}
Let $M$ be a right noetherian monoid, and let $\rho$ be a right congruence on $G\times M.$
Define a right congruence $\rho^{\prime}$ on $M$ by
$$m\,\rho^{\prime}\,n\iff(g, m)\,\rho\,(h, n)\text{ for some }g, h\in G.$$
Since $M$ is right noetherian, $\rho^{\prime}$ is generated by some finite set $X$.
For each element $(x, y)\in\overline{X}$, choose $\alpha(x, y), \beta(x, y)\in G$ such that 
$$(\alpha(x, y), x)\,\rho\,(\beta(x, y), y),$$
and let 
$$Y=\bigl\{\bigl((\alpha(x, y), x), (\beta(x, y), y)\bigr) : (x, y)\in\overline{X}\bigr\}.$$
By Lemma \ref{dpf}, there exists a finite set $S$ such that, for every $m\in M,$ 
we have $m=sm^{\prime}$ for some $m^{\prime}\in M$ and $\rho_m=\rho_s.$
Since $G$ is noetherian, each $\rho_s$ is generated by some finite set $U_s$.
For $s\in S$, let $$V_s=\bigl\{\bigl((u, s), (v, s)\bigr) : (u, v)\in U_s\bigr\}.$$
We shall show that $\rho$ is generated by the finite set 
$$Z=Y\cup\biggl(\bigcup_{s\in S}V_s\biggr).$$
Let $(g, m)\,\rho\,(h, n)$ with $(g, m)\neq(h, n)$.
Assume that $m\neq n$.
Since $m\,\rho^{\prime}\,n$, there exists an $X$-sequence
$$m=x_1m_1, y_1m_1=x_2m_2, \dots, y_km_k=n,$$
where $(x_i, y_i)\in\overline{X}$ and $m_i\in M$ for $1\leq i\leq k.$
Let $\alpha_i=\alpha(x_i, y_i)$ and $\beta_i=\beta(x_i, y_i)$ for $i\in\{1, \dots, k\}$, and let
$$g^{\prime}=\beta_k\alpha_k^{-1}\dots\beta_1\alpha_1^{-1}g.$$
We then have a $Y$-sequence
\begin{equation*}
\begin{split}
(g, m)=(\alpha_1, x_1)(\alpha_1^{-1}g, m_1),\; &(\beta_1, y_1)(\alpha_1^{-1}g, m_1)=(\alpha_2, x_2)(\alpha_2^{-1}\beta_1\alpha_1^{-1}g, m_2),\\
&\hspace{1.8em}\dots, (\beta_k, y_k)(\alpha_k^{-1}\dots\beta_1\alpha_1^{-1}g, m_k)=(g^{\prime}, n).
\end{split}
\end{equation*}
If $g^{\prime}=h$, then we are done, so assume that $g^{\prime}\neq h$.
(If $m=n$, simply let $g^{\prime}=g$).\par
Now, there exist $s\in S$ and $m^{\prime}\in M$ such that $n=sm^{\prime}$ and $\rho_n=\rho_s$.
It follows that $g^{\prime}\,\rho_s\,h$, and hence there exists a $U_s$-sequence
$$g^{\prime}=u_1g_1, v_1g_1=u_2g_2, \dots, v_lg_l=h.$$
Therefore, we have a $V_s$-sequence
$$(g^{\prime}, n)=(u_1, s)(g_1, m^{\prime}), (v_1, s)(g_1, m^{\prime})=(u_2, s)(g_2, m^{\prime}), \dots, (v_l, s)(g_l, m^{\prime})=(h, n).$$
Hence, $\bigl((g, m), (h, n)\bigr)$ is a consequence of $Z.$
\end{proof}

The next lemma is concerned with {\em Brandt semigroups}, which are completely $0$-simple inverse semigroups.
They can also be characterised as follows.\par
Let $G$ be a group and let $I$ be a non-empty set.
We define the following multiplication on $(I\times G\times I)\cup\{0\}:$
$$(i, g, j)(k, h, l)=
\begin{cases} 
   (i, gh, l) & \text{if }j=k\\
   0 & \text{otherwise.}
\end{cases}$$
The set $(I\times G\times I)\cup\{0\}$ is a semigroup under this multiplication, and we denote it by $B(G, I).$
Brandt semigroups are precisely the semigroups isomorphic to some $B(G, I)$ \cite[Theorem II.3.5]{Petrich}.

\begin{lemma}
\label{dpBrandt}
Let $G$ be a noetherian group and let $I$ be a finite set.   
Then the Brandt semigroup $S=B(G, I)$ preserves right noetherian in direct products.
\end{lemma}

\begin{proof}
Given Lemma \ref{dpgroup}, we may now assume that $|I|\geq 2.$
Let $M$ be a right noetherian monoid, and let $\rho$ be a right congruence on $S\times M.$
For each $i\in I,$ let $G_i$ denote the subgroup $\{i\}\times G\times\{i\}\cong G,$ and let $\rho_i$ be the restriction of $\rho$ to $G_i\times M.$
Also, let $\rho_0$ denote the restriction of $\rho$ to $\{0\}\times M.$
By Lemma \ref{dpgroup}, we have that each $\rho_i, i\in I\cup\{0\},$ is generated by a finite set $X_i.$\par
Fix $i_0\in I.$  Consider the set
$$P=\{(i, j)\in I\times I : i\neq j, \bigl((i, 1, i_0), m\bigr)\,\rho\,\bigl((j, g, i_0), n\bigr)\text{ for some }g\in G, m, n\in M\}.$$
For $(i, j)\in P$, let $J(i, j)$ denote the set
$$\{m\in M : \bigl((i, 1, i_0), m\bigr)\,\rho\,\bigl((j, g, i_0), n\bigr)\text{ for some }g\in G, n\in M\}.$$
It is clear that $J(i, j)$ is a right ideal of $M$.
Since $M$ is weakly right noetherian, $J(i, j)$ is generated (as a right ideal) by some finite set $A(i, j)$.
For each $a\in A(i, j)$, choose $\alpha_{i, j}(a)\in G$ and $\beta_{i, j}(a)\in M$ such that 
$$\bigl((i, 1, i_0), a\bigr)\,\rho\,\bigl((j, \alpha_{i, j}(a), i_0), \beta_{i, j}(a)\bigr).$$
Let $H$ be the finite set 
$$\Bigl\{\Bigl(\bigl((i, 1, i_0), a\bigr), \bigl((j, \alpha_{i, j}(a), i_0), \beta_{i, j}(a)\bigr)\Bigr) : (i, j)\in P, a\in A(i, j)\Bigr\}.$$
Now consider the set $$R=\{i\in I : \bigl((i, 1, i_0), m\bigr)\,\rho\,(0, n)\text{ for some }m, n\in M\}.$$
For $i\in R,$ let $J_i$ denote the set 
$$\{m\in M : \bigl((i, 1, i_0), m\bigr)\,\rho\,(0, n)\text{ for some }n\in M\}.$$
Now $J_i$ is a right ideal of $M,$ so it is generated by some finite set $A_i$.
For each $a\in A_i$, choose $\alpha_i\in M$ such that $\bigl((i, 1, i_0), a\bigr)\,\rho\,(0, \alpha_i)$,
and let $K$ be the finite set 
$$\Bigl\{\Bigl(\bigl((i, 1, i_0), a\bigr), (0, \alpha_i)\Bigr) : i\in R, a\in A_i\Bigr\}.$$
We claim that $\rho$ is generated by the finite set 
$$X=\Bigl(\bigcup_{i\in I\cup\{0\}}X_i\Bigr)\cup H\cup K.$$
Let $u=(s, m)$and $v=(t, n)$, where $s=(i, g, k)$ and $t=(j, h, l),$ be such that $u\,\rho\,v$ and $u\neq v.$\par
Suppose first that $k=l$.  
Multiplying on the right by $\bigl((k, g^{-1}, j), 1)\bigr)$, we have that
$$u^{\prime}=\bigl((i, 1, j), m)\bigr)\,\rho\,\bigl((j, hg^{-1}, j), n\bigr)=v^{\prime}.$$
We show that $(u^{\prime}, v^{\prime})$ is a consequence of $X.$
Then, multiplying on the right by $\bigl((j, g, k), 1)\bigr),$
we have that $(u, v)$ is a consequence of $X.$\par
If $i=j,$ then $(u^{\prime}, v^{\prime})$ is a consequence of $X_i.$\par
If $i\neq j,$ we have that $m\in J(i, j),$ so $m=am^{\prime}$ for some $a\in A(i, j)$ and $m^{\prime}\in M.$
Now, applying an element of $H$ to
$$u^{\prime}=\bigl((i, 1, i_0), a\bigr)\bigl((i_0, 1, i), m^{\prime}\bigr),$$
we obtain $$u^{\prime\prime}=\bigl((j, \alpha_{i, j}(a), i), \beta_{i, j}(a)\bigr),$$
and we have that $(u^{\prime\prime}, v^{\prime})$ is a consequence of $X_j.$\par
Now suppose that $k\neq l.$
Multiplying on the right by $\bigl((k, g^{-1}, i_0), 1)\bigr)$, we have that
$$\bigl((i, 1, i_0), m)\bigr)\,\rho\,(0, n).$$
We have that $m\in A_i,$ so $m=ap$ for some $a\in A_i$ and $p\in M.$
Now, applying an element of $K$ to 
$$u=\bigl((i, 1, i_0), a\bigr)\bigl((i_0, g, k), p\bigr),$$ we obtain $u^{\prime}=(0, a_ip).$
Similarly, there exists $n^{\prime}\in M$ such that $v^{\prime}=(0, n^{\prime})$ is obtained from $v$ by an element of $K.$
Since $u^{\prime}\,\rho_0\,v^{\prime},$ we have that $(u^{\prime}, v^{\prime})$ is a consequence of $X_0.$
Hence, $(u, v)$ is a consequence of $X.$
\end{proof}

A {\em principal series} of a semigroup $S$ is a finite chain of ideals
$$K(S)=I_1\subset I_2\subset\dots I_n=S,$$
where $K(S)$ is the unique minimal ideal of $S,$ called the {\em kernel} of $S,$
and, for each $k\in\{1, \dots, n-1\},$ the ideal $I_k$ is maximal in $I_{k+1}.$\par
We are now in a position to prove Theorem \ref{dpinverse}.

\begin{proof}[Proof of Theorem \ref{dpinverse}.]
It follows from Theorem \ref{inverse} that every right noetherian inverse semigroup has a principal series.
We shall prove Theorem \ref{dpinverse} by induction on the length of the principal series.
We may assume that $S$ has a zero $0$ by Corollary \ref{dpzero},
so let $S$ have a principal series $$\{0\}\subset I_1\subset I_2\subset\dots\subset I_n=S.$$
Now $I_1$ is completely $0$-simple by \cite[Lemma 1.3]{Kozhukhov1}, so it is a Brandt semigroup $B(G, I).$
It follows from Theorem \ref{inverse} that $G$ is noetherian and $I$ is finite, 
so $I_1$ preserves right noetherian in direct products by Lemma \ref{dpBrandt}.\par
Suppose now that $n>1,$ and assume that any right noetherian inverse semigroup with a shorter principal series preserves right noetherian in direct products.
The Rees quotient $S/I_1$ is right noetherian by Lemma \ref{quotient},
and it preserves right noetherian in direct products by the inductive hypothesis.
It now follows from Lemma \ref{dpidealextension} that $S$ preserves right noetherian in direct products.
\end{proof}

\begin{remark}
Example \ref{dpex} shows that Theorem \ref{dpinverse} does not generalise to regular semigroups.
\end{remark}

\section{Semidirect and wreath products\nopunct\\}

In this section we consider semidirect products of semigroups and wreath products of monoids.\par
Let $S$ and $T$ be two semigroups.
A {\em (left) action} of $T$ on $S$ is a map 
$$T\times S\to S, (t, s)\mapsto t\cdot s$$
such that, for all $s, s_1, s_2\in S$ and $t, t_1, t_2\in T$,
\begin{enumerate}
 \item $t_1\cdot(t_2\cdot s)=(t_1t_2)\cdot s$,
 \item $t\cdot(s_1s_2)=(t\cdot s_1)(t\cdot s_2)$.
\end{enumerate}
If $T$ is a monoid with identity $1,$ we also require that $1\cdot s=s$ for all $s\in S.$\par
The {\em semidirect product} of $S$ and $T$ (with respect to the given action) is the Cartesian product $S\times T$ with multiplication given by
$$(s_1, t_1)(s_2, t_2)=(s_1(t_1\cdot s_2), t_1t_2)$$
for all $s_1, s_2\in S$ and $t_1, t_2\in T$.\par
As observed in Section 4, a direct product being right noetherian implies that the factors are right noetherian.
We have the following generalisation of this fact.

\begin{prop}
\label{sdpfactors}
Let $S$ and $T$ be two semigroups, and let $U$ be a semidirect product of $S$ and $T$.
If $U$ is right noetherian, then both $S$ and $T$ are right noetherian.
\end{prop}

\begin{proof}
Since $T$ is a homomorphic image of $U$, it follows from Lemma \ref{quotient} that $T$ is right noetherian.\par
Now let $\rho$ be a right congruence on $S.$
We define a right congruence $\rho^{\prime}$ on $U$ by
$$(s_1, t_1)\,\rho^{\prime}\,(s_2, t_2)\iff s_1\,\rho\,s_2\text{ and }t_1=t_2.$$
Since $U$ is right noetherian, $\rho^{\prime}$ is generated by a finite set $X$.
We claim that $\rho$ is generated by the finite set 
$$Y=\{(x, y) : \bigl((x, t), (y, t)\bigr)\in X\text{ for some }t\in T\}.$$
Let $a\,\rho\,b$ with $a\neq b$.
Choose $t\in T$.  Then $(a, t)\,\rho^{\prime}(b, t)$, so there exists an $X$-sequence
$$(a, t)=(x_1, t_1)(s_1, u_1), (y_1, t_1)(s_1, u_1)=(x_2, t_2)(s_2, u_2), \dots, (y_k, t_k)(s_k, u_k)=(b, t).$$
Hence, we have a $Y$-sequence
$$a=x_1(t_1\cdot s_1), y_1(t_1\cdot s_1)=x_2(t_2\cdot s_2), \dots, y_k(t_k\cdot s_k)=b,$$
so $(a, b)$ is a consequence of $Y.$
\end{proof}

We now define the wreath product of two monoids; see \cite{Thomson} for more details.\par
Let $M$ and $N$ be two monoids.
Let $M^N$ denote the set of all functions from $N$ to $M.$
For $m\in M,$ we let $c_m$ denote the map in $M^N$ that maps every element of $N$ to $m.$
Under componentwise multiplication, $M^N$ is a monoid with identity $c_1.$\par
For a function $f\in M^N,$ the {\em support} of $f$ is the set $$\{n\in N : nf\neq 1_M\}.$$
The set $M^{(N)}$ of all $f\in M^N$ with finite support is a submonoid of $M^N.$
Note that $M^{(N)}$ and $M^N$ are identical if $N$ is finite.\par
We define an action of $N$ on $M^N$ as follows:
Given a function $f\in M^N$ and an element $n\in N,$ 
we define a function $^n\!f\in M^N$ by $^n\!f=(xn)f$ for all $x\in N.$
This action yields a semidirect product of $M^N$ and $N,$ called the {\em unrestricted wreath product} of $M$ and $N,$ which we denote by $M\,wr\,N.$\par 
The {\em (restricted) wreath product} of $M$ and $N,$ denoted by $M\wr N,$ is the subsemigroup of $M\,wr\,N$ generated by $M^{(N)}\times N.$
It is well known that if $N$ is a group, then $M\wr N=M^{(N)}\times N.$
Also, if $N$ is finite, then $M\wr N=M\,wr\,N.$\par
The next result provides some necessary conditions for the wreath product of two monoids to be right noetherian.

\begin{prop}
\label{wrfactors}
Let $M$ and $N$ be two monoids.  
If the wreath product $M\wr N$ is right noetherian, then one of the following holds:
\begin{enumerate}
 \item $M$ is trivial and $N$ is right noetherian;
 \item $M$ is right noetherian and $N$ is finite.
\end{enumerate}
\end{prop}

\begin{proof}
If $M$ is trivial, then $N\cong M\wr N$ is right noetherian.\par
Suppose that $M$ is non-trivial.
To prove that $N$ is finite, we assume for a contradiction that it is infinite.
Let $1$ and $e$ be the identities of $M$ and $N$ respectively.  
Choose $m\in M\!\setminus\!\{1\}$ and a countably infinite set $\{n_1, n_2, \dots\}\subseteq N.$
For each $i\in\mathbb{N},$ we define a map
$$f_i : N\to M, n\mapsto
\begin{cases} 
   m & \text{if } n\in\{n_1, \dots, n_i\}\\
   1 & \text{otherwise.}
  \end{cases}$$
Let $\rho$ be the right congruence on $M\wr N$ generated by the set 
$$\bigl\{\bigl((f_i, e), (c_1, e)\bigr) : i\in\mathbb{N}\bigr\}.$$
Since $M\wr N$ is right noetherian, $\rho$ can be generated by some finite set 
$$X=\bigl\{\bigl((f_i, e), (c_1, e)\bigr) : 1\leq i\leq l\bigr\}.$$
Choose $i>l,$ and let $f=f_i$ and $n=n_i.$
We have an $X$-sequence
$$(f, e)=(x_1, e)(g_1, e), (y_1, e)(g_1, e)=(x_2, e)(g_2, e), \dots, (y_k, e)(g_k, e)=(c_1, e).$$
Therefore, we have a sequence
$$f=x_1g_1, y_1g_1=x_2g_2, \dots, y_kg_k=c_1.$$
Since $nf=m$ and $nx_1=1,$ we have $ng_1=m.$
Therefore, we have $n(y_1g_1)=m,$ which in turn implies that $ng_2=m.$
Continuing in this way, we have that $ng_i=m$ for all $i\in\{1, \dots, k\}.$
But then we have $$nc_1=(ny_k)(ng_k)=m\neq 1,$$ which is a contradiction.\par 
Since $N$ is finite, $M\wr N$ coincides with $M\,wr\,N.$
It follows from Proposition \ref{sdpfactors} that the monoid $M^N$ is right noetherian,
and hence $M,$ being a homomorphic image of $M^N,$ is also right noetherian.
\end{proof}

The following example shows that the converse of Proposition \ref{wrfactors} does not hold,
and the property of being right noetherian is not preserved under semidirect products of monoids.

\begin{ex}
Let $M$ be an infinite right noetherian monoid such that $M\!\setminus\!\{1_M\}$ is a subsemigroup,
and let $N$ be any finite monoid with a zero $0.$
Note that if $M$ preserves right noetherian in direct products (for example, $M=\mathbb{N}_0$), then $M^N$ is right noetherian.
Consider the ideal $$I=\{(f, 0) : f\in M^N\}$$ of $M\wr N.$
Let $X\times\{0\}$ be any set that generates $I$ as a right ideal,
and let $S$ denote the infinite set $$\{f\in M^N : 1_Nf=1_M\}.$$
We claim that $S$ is contained in $X,$ so $I$ is not finitely generated as a right ideal and $M\wr N$ is not weakly right noetherian.
Indeed, if $f\in S,$ then $(f, 0)=(g, 0)(h, n)$ for some $g\in X,$ $h\in M^N$ and $n\in N,$ so $f=g\,^0\!h=gc_{0h}.$ 
Therefore, we have 
$$1_M=1_Nf=(1_Ng)(1_Nc_{0h})=(1_Ng)(0h),$$
which implies that $0h=1_M$ (since $M\!\setminus\!\{1_M\}$ is an ideal of $M$) and hence $f=g.$
\end{ex}

For groups, the property of being noetherian \textit{is} preserved under semidirect products:

\begin{thm}
\label{sdpgroups}
Let $G$ and $H$ be two groups, and let $U$ be a semidirect product of $G$ and $H$.
Then $U$ is noetherian if and only if both $G$ and $H$ are noetherian.
\end{thm}

\begin{proof}
The direct implication follows from Proposition \ref{sdpfactors}.
For the converse, we need to show that every subgroup of $U$ is finitely generated, so let $S$ be a subgroup of $U.$
Define a set $$G_S=\{g\in G : (g, 1)\in S\}.$$
Since $G_S$ is a subgroup of $G$ and $G$ is noetherian, we have that $G_S$ is generated by some finite set $X.$\par 
Let $H_S$ be the projection of $S$ onto $H$.
Since $H_S$ is a subgroup of $H$ and $H$ is noetherian, we have that $H_S$ is generated by some finite set $Y.$
For each $y\in Y,$ choose $g_y\in G$ such that $(g_y, y)\in S.$
We claim that $S$ is generated by the finite set 
$$Z=\{(x, 1) : x\in X\}\cup\{(g_y, y) : y\in Y\}\subseteq S.$$
Let $(g, h)\in S$.  Now $h=y_1\dots y_k$ for some $y_i\in Y\cup Y^{-1}.$
For $i\in\{1, \dots, k\},$ let $u_i=g_{y_i}$ if $y_i\in Y$ or $u_i=y_i\cdot g_{{y_i}^{-1}}^{-1}$ if $y_i\in Y^{-1}.$
We have that
$$(g, h)(u_k, y_k)^{-1}\dots(u_1, y_1)^{-1}=(g^{\prime}, 1)\in S,$$
where $$g^{\prime}=g(y_1\dots y_{k-1}\cdot u_k^{-1})\dots(y_1\cdot u_2^{-1})u_1^{-1}.$$
Since $g^{\prime}\in G_S,$ we have $g^{\prime}=x_1\dots x_n$ for some $x_i\in X\cup X^{-1}.$
It now follows that
$$(g, h)=(x_1, 1)\dots(x_n, 1)(u_1, y_1)\dots(u_k, y_k)\in\langle Z\rangle,$$
as required.
\end{proof}

We now provide necessary and sufficient conditions for the wreath product of two groups to be noetherian.

\begin{thm}
\label{wrgroups}
Let $G$ and $H$ be two groups.  Then the wreath product $G\wr H$ is noetherian if and only if one of the following holds:
\begin{enumerate}
 \item $G$ is trivial and $H$ is noetherian;
 \item $G$ is noetherian and $H$ is finite.
\end{enumerate}
\end{thm}

\begin{proof}
The direct implication follows from Proposition \ref{wrfactors}, so we just need to prove the converse.\par
If $G$ is trivial and $H$ is noetherian, then $G\wr H\cong H$ is noetherian.
Suppose that $G$ is noetherian and $H$ is finite.
The wreath product of $G$ and $H$ is a semidirect product of $G^H$ and $H,$ 
and $G^H$ is isomorphic to the direct product of $|H|$ copies of $G,$
so it follows from Theorem \ref{sdpgroups} that $G\wr H$ is noetherian.
\end{proof}

\section{Unions of semigroups\nopunct}

In this section we study $0$-direct unions and (strong) semilattices of semigroups.
We begin with $0$-direct unions.

\begin{defn}
Let $U$ be a semigroup $S\cup T\cup 0$, where $S$ and $T$ are disjoint subsemigroups of $U$, $0$ is a zero element disjoint from $S\cup T,$
and $st=ts=0$ for all $s\in S$ and $t\in T.$
Then $U$ is the {\em $0$-direct union} of $S$ and $T.$
\end{defn}

The property of being right noetherian is preserved under $0$-direct unions:

\begin{prop}
Let $U$ be the $0$-direct union of two semigroups $S$ and $T$.
Then $U$ is right noetherian if and only if both $S$ and $T$ are right noetherian.
\end{prop}

\begin{proof}
($\Rightarrow$) Since $U$ is right noetherian and the sets $U\!\setminus\!S$ and $U\!\setminus\!T$ are ideals of $U,$
we have that $S$ and $T$ are right noetherian by Proposition \ref{leftideal}.\par
($\Leftarrow$) Since $T$ is right noetherian, the ideal $U\!\setminus\!S=T\cup\{0\}$ is right noetherian by Corollary \ref{adjoinzero}.
Therefore, since $S$ is right noetherian, we have that $U$ is right noetherian by Corollary \ref{rnideal}.
\end{proof}

We now turn our attention to semilattices of semigroups.\par 
Let $Y$ be a semilattice and let $(S_{\alpha})_{\alpha\in Y}$ be a family of disjoint semigroups (resp. monoids), indexed by $Y,$
such that $S=\bigcup_{\alpha\in Y}S_{\alpha}$ is a semigroup.
If $S_{\alpha}S_{\beta}\subseteq S_{\alpha\beta}$ for all $\alpha, \beta\in Y,$
then $S$ is called a {\em semilattice of semigroups (resp. monoids)}, and we denote it by $S=\mathcal{S}(Y, S_{\alpha}).$\par 
Now suppose that for each $\alpha, \beta\in Y$ with $\alpha\geq\beta,$ 
there exists a homomorphism $\phi_{\alpha, \beta} : S_{\alpha}\to S_{\beta}$ satisfying the following:
\begin{itemize}
 \item for each $\alpha\in Y,$ the homomorphism $\phi_{\alpha, \alpha}$ is the identity map on $S_{\alpha}$;
 \item for each $\alpha, \beta, \gamma\in Y$ with $\alpha\geq\beta\geq\gamma$, we have $\phi_{\alpha, \beta}\,\phi_{\beta, \gamma}=\phi_{\alpha, \gamma}.$
\end{itemize}
For $a\in S_{\alpha}$ and $b\in S_{\beta},$ we define
$$ab=(a\phi_{\alpha, \alpha\beta})(b\phi_{\beta, \alpha\beta}).$$
With this multiplication, $S$ is a semilattice of semigroups (resp. monoids).
In this case we call $S$ a {\em strong semilattice of semigroups} (resp. {\em monoids)} and denote it by $S=\mathcal{S}(Y, S_{\alpha}, \phi_{\alpha, \beta}).$\par 
In the remainder of this section, we investigate under what conditions a (strong) semilattice of semigroups or monoids is right noetherian.\par
Since the semilattice $Y$ is a homomorphic image of $S=\mathcal{S}(Y, S_{\alpha}),$ 
by Lemma \ref{quotient} and Corollary \ref{semilattice} we have:

\begin{lemma}
\label{finitesemilattice}
Let $S=\mathcal{S}(Y, S_{\alpha})$ be a semilattice of semigroups.
If $S$ is right noetherian, then $Y$ is finite.
\end{lemma}

For a semilattice of semigroups $\mathcal{S}(Y, S_{\alpha})$ to be right noetherian,
it is not required that all the $S_{\alpha}$ be right noetherian.
In order to show this, we consider a construction involving semigroup acts.

\begin{con}
Let $S$ be a semigroup, and let $A$ be an $S$-act disjoint from $S$ with action 
$$A\times S\to A, (a, s)\mapsto a{\cdot}s.$$  
We define the following multiplication on the set $S\cup A$:
$$x\circ y=
  \begin{cases} 
   xy & \text{if }x, y\in S\\
   x{\cdot}y & \text{if }x\in A, y\in S\\
   y & \text{if }x\in S\cup A, y\in A.
  \end{cases}$$
With this operation, the set $S\cup A$ is a semilattice of two semigroups $S_{\alpha}=S$ and $S_{\beta}=A$ where $\alpha>\beta$, 
and we denote it by $\mathcal{U}(S, A).$
Note that $A$ is an ideal of $\mathcal{U}(S, A)$.
\end{con}

\begin{ex}
Let $A$ be a free cyclic $\mathbb{Z}$-act disjoint from $\mathbb{Z}.$  
Then $\mathcal{U}(\mathbb{Z}, A)$ is the monoid defined by the presentation 
$$\langle a, b, c\,|\,ab=ba=1, ac=c^2=c\rangle.$$
\end{ex}

Notice that $A$ is a right zero subsemigroup of $\mathcal{U}(S, A).$
It is easy to see that infinite right zero semigroups are not right noetherian, 
so the following result yields the desired counterexample.

\begin{prop}
Let $S$ be a semigroup, let $A$ be an $S$-act disjoint from $S,$ and let $U=\mathcal{U}(S, A).$
Then $U$ is right noetherian if and only if $S$ is right noetherian and $A$ is a finitely generated $S$-act.
\end{prop}

\begin{proof}
($\Rightarrow$) Since $U\!\setminus\!S=A$ is an ideal of $U,$ we have that $S$ is right noetherian by Proposition \ref{leftideal}.\par 
Let $\rho$ be the right congruence on $U$ generated by $A\times A.$
Since $U$ is right noetherian, $\rho$ is generated by some finite set $Y\subseteq A\times A.$
We claim that $A$ is generated by the finite set 
$$X=\{x\in A : (x, y)\in\overline{Y}\text{ for some }y\in A\}.$$
Let $a\in A,$ and choose $b\in A\!\setminus\!\{a\}.$  Since $a\,\rho\,b,$ there exists a $Y$-sequence
$$a=x_1u_1, y_1u_1, \dots, y_ku_k=b$$
of shortest possible length connecting $a$ and $b.$
If $u_1\in A,$ then $a=u_1=y_1u_1,$ which is a contradiction.
Hence, we have $u_1\in S^1$ and $a=x_1u_1\in XS^1.$\par
($\Leftarrow$) Let $\rho$ be a right congruence on $U,$ and let $X$ be a finite generating set for $A.$
Let $Q(X)$ denote the set of $x\in X$ such that $x{\cdot}s\,\rho\,t$ for some $s, t\in S.$
For each $x\in Q(X),$ we define a right ideal 
$$I_x=\{s\in S : x{\cdot}s\,\rho\,t\text{ for some }t\in S\}.$$
We have that $S$ is weakly right noetherian by Lemma \ref{weakly}, so $I_x$ is generated as a right ideal by some finite set $P_x.$
For each $p\in P_x,$ choose $\alpha_x(p)\in S$ such that $x{\cdot}s\,\rho\,\alpha_x(p),$ and let 
$$H=\{(x{\cdot}p, \alpha_x(p))\in A\times S : x\in Q(X), p\in P_x\}.$$ 
Now, let $R(X)$ denote the set of pairs $(x, y)\in X\times X$ with $x\neq y$ such that $x{\cdot}s\,\rho\,y{\cdot}t$ for some $s, t\in S.$
For each pair $(x, y)\in R(X),$ we define a right ideal
$$I(x, y)=\{s\in S : x{\cdot}s\,\rho\,y{\cdot}t\text{ for some }t\in S\}.$$
We have that $I(x, y)$ is generated by some finite set $P(x, y).$
For each $p\in P(x, y),$ choose $\alpha_{x, y}(p)\in S$ such that $x{\cdot}s\,\rho\,y{\cdot}\alpha_{x, y}(p),$
and let 
$$K=\{(x{\cdot}p, y{\cdot}\alpha_{x, y}(p))\in A\times A : (x, y)\in R(X), p\in P(x, y)\}.$$
Now, let $\rho_S$ be the restriction of $\rho$ to $S$.
Since $S$ is right noetherian, we have that $\rho_S$ is generated by some finite set $Y.$\par
For each $x\in X,$ we define a right congruence $\rho_x$ on $S$ by
$$s\,\rho_x\,t\iff x{\cdot}s\,\rho\,x{\cdot}t.$$
We have that $\rho_x$ is generated by some finite set $Y_x.$
Let $$Z_x=\{(x{\cdot}y, x{\cdot}y^{\prime})\in A\times A : (y, y^{\prime})\in Y_x\}.$$
We claim that $\rho$ is generated by the finite set 
$$Z=H\cup K\cup Y\cup\biggl(\bigcup_{x\in X}Z_x\biggr).$$
Let $u\,\rho\,v$ with $u\neq v.$  There are three cases.\par
\textit{Case 1}: $u, v\in S.$  Since $u\,\rho_S\, v,$ we have that $(u, v)$ is a consequence of $Y.$\par
\textit{Case 2}: $u\in A,$ $v\in S.$
We have that $u=x{\cdot}s$ for some $x\in X$ and $s\in S.$
Since $s\in I_x,$ we have that $s=ps^{\prime}$ for some $p\in P_x$ and $s^{\prime}\in S.$
We obtain $t=\alpha_x(p)s^{\prime}\in S$ from $u$ using $(x{\cdot}p, \alpha_x(p))\in H,$
and we have that $(t, v)$ is a consequence of $Y.$\par
\textit{Case 3}: $u, v\in A.$
We have that $u=x{\cdot}s$ and $v=y{\cdot}t$ for some $x, y\in X$ and $s, t\in S.$
If $x=y,$ then $s\,\rho_x\,t,$ so there exists a $Y_x$-sequence connecting $s$ and $t.$
Therefore, there clearly exists a $Z_x$-sequence connecting $u$ and $v.$\par 
Now suppose that $x\neq y.$
Since $s\in I(x, y),$ we have that $s=pr$ for some $p\in P(x, y)$ and $r\in S.$
Letting $s^{\prime}=\alpha_{x, y}(p)r\in S,$ we obtain $y{\cdot}s^{\prime}$ from $u$ using $(x{\cdot}p, y{\cdot}\alpha_{x, y}(p))\in K,$
and we have that $(y{\cdot}s^{\prime}, v)$ is a consequence of $Z_y.$
\end{proof}

\begin{remark}
The semigroup $\mathcal{U}(S, A)$ can also be viewed as the disjoint union of $S$ and infinitely many copies of the trivial semigroup.
Therefore, the finiteness of $I$ is not necessary for a disjoint union of semigroups $S=\bigcup_{i\in I}S_i$ to be right noetherian.
\end{remark}

\begin{remark}
If $S$ is regular, then $\mathcal{U}(S, A)$ is regular, 
so right noetherian regular semigroups can have infinitely many idempotents,
and Proposition \ref{inversesubsemigroup} does not generalise to regular subsemigroups.
\end{remark}

In the case that a semigroup $S_{\beta},$ $\beta\in Y,$ is actually a monoid, 
a necessary condition for $S=\mathcal{S}(Y, S_{\alpha})$ to be right noetherian is that $S_{\beta}$ is right noetherian.

\begin{prop}
\label{sofsmonoid}
Let $S=\mathcal{S}(Y, S_{\alpha})$ be a semilattice of semigroups, let $\beta\in Y,$ and suppose that $S_{\beta}$ is a monoid. 
If $S$ is right noetherian, then $S_{\beta}$ is right noetherian.
\end{prop}

\begin{proof}
Let $1_{\beta}$ be the identity of $S_{\beta}.$
Now let $\rho$ be a right congruence on $S_{\beta}.$
Denote by $\overline{\rho}$ the right congruence on $S$ generated by $\rho.$
Since $S$ is right noetherian, $\overline{\rho}$ is generated by some finite set $X\subseteq\rho.$
We claim that $\rho$ is generated by $X.$\par
Indeed, let $a\,\rho\,b$ with $a\neq b.$  Then $a\,\overline{\rho}\,b,$ so there exists an $X$-sequence
$$a=x_1s_1, y_1s_1=x_2s_2, \dots, y_ks_k=b,$$
where $(x_i, y_i)\in\overline{X}$ and $s_i\in S^1$ for $1\leq i\leq k.$
Therefore, we have a sequence
$$a=x_1(1_{\beta}s_1), y_1(1_{\beta}s_1)=x_2(1_{\beta}s_2), \dots, y_k(1_{\beta}s_k)=b.$$
We show that $1_{\beta}s_i\in S_{\beta}$ for all $i\in\{1, \dots, k\},$
so that $(a, b)$ is a consequence of $X,$ as required.
So, for each $i\in\{1, \dots, k\},$ let $s_i\in S_{\alpha_i}^1$ where $\alpha_i=\beta$ if $s_i=1.$
Since $a\in S_{\beta}$ and $a=x_1s_1\in S_{\beta\alpha_1}$, we have that $\beta\alpha_1=\beta.$
Since $x_2s_2\in S_{\beta\alpha_2}$ and $$x_2s_2=y_1s_1\in S_{\beta\alpha_1}=S_{\beta},$$ 
it follows that $\beta\alpha_2=\beta.$
Continuing in this way, we have that $\beta\alpha_i=\beta,$ 
and hence $1_{\beta}s_i\in S_{\beta},$ for all $i\in\{1, \dots, k\}.$
\end{proof}

We have shown that a semilattice of semigroups $\mathcal{S}(Y, S_{\alpha})$ being right noetherian does not imply that every $S_{\alpha}$ is right noetherian.
In the opposite direction, however, we have the following result.

\begin{prop}

\label{sofs}
Let $S=\mathcal{S}(Y, S_{\alpha})$ be a semilattice of semigroups.
If $Y$ is finite and each $S_{\alpha}$ is right noetherian, then $S$ is also right noetherian.
\end{prop}

\begin{proof}
Let $\rho$ be a right congruence on $S.$
For each $\alpha\in Y,$ let $\rho_{\alpha}$ denote the restriction of $\rho$ to $S_{\alpha}.$
Since $S_{\alpha}$ is right noetherian, we have that $\rho_{\alpha}$ is generated by a finite set $X_{\alpha}.$
Let $Q$ denote the set 
$$\{(\alpha, \beta)\in Y\times Y : \alpha\neq\beta, a\,\rho\,b\text{ for some }a\in S_{\alpha}, b\in S_{\beta}\}.$$
For each pair $(\alpha, \beta)\in Q,$ let $I_{\alpha, \beta}$ denote the right ideal of $S_{\alpha}$ generated by the set
$$U_{\alpha, \beta}=\{a\in S_{\alpha} : a\,\rho\,b\text{ for some }b\in S_{\beta}\}.$$
Since $S_{\alpha}$ is weakly right noetherian, we have that $I_{\alpha, \beta}$ is generated by some finite subset $U_{\alpha, \beta}^{\prime}\subseteq U_{\alpha, \beta}.$
For each $u\in U_{\alpha, \beta}^{\prime},$ choose $\lambda_{\alpha, \beta}(u)\in S_{\beta}$ such that $u\,\rho\,\lambda_{\alpha, \beta}(u).$
We now define a set 
$$H=\{(u, \lambda_{\alpha, \beta}(u)) : (\alpha, \beta)\in Q, u\in U_{\alpha, \beta}^{\prime}\}.$$
We claim that $\rho$ is generated by the finite set 
$$Y=\biggl(\bigcup_{\alpha\in Y}X_{\alpha}\biggr)\cup H.$$
Indeed, let $a\,\rho\,b$ with $a\in S_{\alpha},$ $b\in S_{\beta}$ and $a\neq b.$  Set $\gamma=\alpha\beta=\beta\alpha.$\par
We claim that there exist $c\in S_{\beta}\cup S_{\gamma}$ and $d\in S_{\alpha}\cup S_{\gamma}$ such that $(a, c)$ and $(b, d)$ are consequences of $H.$
Indeed, assuming that $\alpha\neq\gamma$ (if $\alpha=\gamma,$ simply let $a=c$),
we have that $a\in I_{\alpha, \beta},$ so there exist $u\in U_{\alpha, \beta}^{\prime}$ and $s\in S_{\alpha}^1$ such that $a=us.$
Now, applying the generating pair $(u, \lambda_{\alpha, \beta}(u))\in H,$ we obtain $c=\lambda_{\alpha, \beta}(u)s\in S_{\beta}\cup S_{\gamma}.$
A similar argument proves the existence of $d.$\par
There are now two cases to consider.\par
\textit{Case 1}: $c\in S_{\beta}$ or $d\in S_{\alpha}.$
If $s\in S_{\beta},$ then $b\,\rho_{\beta}\,c,$ so $(b, c)$ is a consequence of $X_{\beta},$ and hence $(a, b)$ is a consequence of $H\cup X_{\beta}.$
Likewise, if $d\in S_{\alpha},$ then $(a, b)$ is a consequence of $H\cup X_{\alpha}.$\par
\textit{Case 2}: $c, d\in S_{\gamma}.$  
We have that $c\,\rho_{\gamma}\,d,$ so that $(c, d)$ is a consequence of $X_{\gamma},$ and hence $(a, b)$ is a consequence of $H\cup X_{\gamma}.$
\end{proof}

From Propositions \ref{sofsmonoid} and \ref{sofs} we obtain:

\begin{corollary}
\label{sofm}
A semilattice of monoids $S=\mathcal{S}(Y, S_{\alpha})$ is right noetherian if and only if $Y$ is finite and each $S_{\alpha}$ is right noetherian.
\end{corollary}

The final result of this section provides necessary and sufficient conditions for a strong semilattice of semigroups to be right noetherian.
\begin{thm}
\label{ssofs}
A strong semilattice of semigroups $S=\mathcal{S}(Y, S_{\alpha}, \phi_{\alpha, \beta})$ is right noetherian if and only if $Y$ is finite and each $S_{\alpha}$ is right noetherian.
\end{thm}

\begin{proof}
The reverse implication follows from Proposition \ref{sofs}.\par
For the direct implication, suppose that $S$ is right noetherian.
We have that $Y$ is finite by Lemma \ref{finitesemilattice}.
To each $S_{\alpha}$ we adjoin an identity $1_{\alpha}.$
For each $\alpha, \beta\in Y$ with $\alpha\geq\beta,$ we extend $\phi_{\alpha, \beta}$
to a homomorphism $\phi_{\alpha, \beta}^{\prime}: S_{\alpha}^{1_{\alpha}}\to S_{\beta}^{1_{\beta}}$ by setting $1_{\alpha}\phi_{\alpha, \beta}^{\prime}=1_{\beta}.$
We now have a strong semilattice of monoids $T=\mathcal{S}(Y, S_{\alpha}^{1_{\alpha}}, \phi_{\alpha, \beta}^{\prime}).$
Since $T\!\setminus\!S$ is finite, we have that $T$ is right noetherian by Therorem \ref{large}.
It now follows from Corollary \ref{sofm} that each $S_{\alpha}^{1_{\alpha}}$ is right noetherian,
and hence each $S_{\alpha}$ is right noetherian by Corollary \ref{adjoinidentity}.
\end{proof}

\section{Free products\nopunct\\}

Let $S$ and $T$ be two semigroups with presentations $\langle X\,|\,R\rangle$ and $\langle Y\,|\,S\rangle$ respectively.
The {\em semigroup free product} of $S$ and $T$ is the semigroup defined by the presentation $\langle X, Y\,|\,R, S\rangle.$
If $S$ and $T$ are monoids, then the {\em monoid free product} of $S$ and $T$ is the monoid defined by the presentation $\langle X, Y\,|\,R, S, 1_S=1_T\rangle.$\par
In the following we provide necessary and sufficient conditions for the semigroup (resp. monoid) free product of two semigroups (resp. monoids) to be right noetherian.

\begin{thm}
\label{sfp}
Let $S$ and $T$ be two semigroups.
Then the semigroup free product of $S$ and $T$ is right noetherian if and only if both $S$ and $T$ are trivial.
\end{thm}

\begin{proof}
Let $U$ be the semigroup free product of $S$ and $T.$\par 
($\Rightarrow$) Suppose that $T$ is non-trivial.  Choose $a\in S$ and $b, c\in B.$
For $i\in\mathbb{N},$ let $u_i=(ab)^iaca.$ 
Let $I$ be the right ideal of $U$ generated by the set $X=\{u_i : i\in\mathbb{N}\}.$
For any $i\in\mathbb{N},$ the element $u_i$ cannot be written as $u_ju$ for any $j\neq i$ and $u\in U,$
so $X$ is a minimal generating set for $I$ and $I$ is not finitely generated. 
Hence, we have that $U$ is not (weakly) right noetherian.\par 
($\Leftarrow$)  The semigroup $U$ has the presentation
$$\langle e, f\,|\,e^2=e, f^2=f\rangle.$$
Let $a_1=ef$, $a_2=fe$, $a_3=efe$ and $a_4=fef$, and let $U_i=\{a_i^n : n\in\mathbb{N}\}$.
Note that each $U_i$ is isomorphic to $\mathbb{N}$.
Let $U^{\prime}=\bigcup_{i=1}^4U_i$.  We shall show that $U^{\prime}$ is right noetherian.
Since $U\!\setminus\!U^{\prime}=\{e, f\}$ is finite, it then follows from Theorem \ref{large} that $U$ is right noetherian.\par
Let $\rho$ be a right congruence on $U^{\prime}$.
For each $i\in\{1, 2, 3, 4\}$, let $\rho_i$ be the restriction of $\rho$ to $U_i$.
Since $U_i$ is right noetherian, $\rho_i$ is generated by some finite set $X_i$.
We now make the following claim.
\begin{claim*}
Let $i, j\in\{1, 2, 3, 4\}$ with $i\neq j$ such that $a_i^M\,\rho\,a_j^N$ for some $M, N\in\mathbb{N}$.
\begin{enumerate}
\item If $i\in\{1, 4\}$ and $j\in\{2, 3\}$, then $a_i^m\,\rho\,a_i^M$ for all $m\geq M$ and $a_j^n\,\rho\,a_j^N$ for all $n\geq N$.
\item If $i=1, j=4$, or $i=2, j=3$, then $a_i^{M+r}\,\rho\,a_j^{N+r}$ for any $r\geq 0$.
\end{enumerate}
\end{claim*}
\begin{proof}
To prove (1), we just consider the case where $i=1$ and $j=2$.
By induction we have that $a_1^m\,\rho\,a_1^M$ for all $m\geq M$:
$$a_1^{M+k}=a_1^Ma_1^k\,\rho\,a_2^Na_1^k=a_2^Nfa_1^{k-1}\,\rho\,a_1^Mfa_1^{k-1}=a_1^{M+k-1}.$$
Similarly, we have that $a_2^n\,\rho\,a_2^N$ for all $n\geq N.$\par 
To prove (2), we just consider the case where $i=1, j=4$.  For $r\geq 0,$ we have
$$a_1^{M+r}=a_1^Ma_1^r\,\rho\,a_4^Na_1^r=a_4^{N+r},$$
as required.
\end{proof}
Let $i, j\in\{1, 2, 3, 4\}$ with $i<j$.
If $a_i^m\,\rho\,a_j^n$ for some $m, n\in\mathbb{N}$, let
$$M_{ij}=\text{min}\{m\in\mathbb{N} : a_i^m\,\rho\,a_j^n\text{ for some }n\in\mathbb{N}\},\,N_{ij}=\text{min}\{n\in\mathbb{N} : a_i^{M_{ij}}\,\rho\,a_j^n\},$$
and let $X_{ij}=\bigl\{(a_i^{M_{ij}}, a_j^{N_{ij}})\bigr\}.$
If there do not exist any $m, n\in\mathbb{N}$ such that $a_i^m\,\rho\,a_j^n$, let $X_{ij}=\emptyset$.
We claim that $\rho$ is generated by the finite set 
$$X=\Biggl(\bigcup_{i=1}^4X_i\Biggr)\cup\Biggl(\bigcup_{1\leq i<j\leq 4}X_{ij}\Biggr).$$
Indeed, let $i, j\in\{1, 2, 3, 4\}$ with $a_i^m\,\rho\,a_j^n$ for some $m, n\in\mathbb{N}$ and $a_i^m\neq a_j^n.$\par
If $i=j$, then $(a_i^m, a_j^n)$ is a consequence of $X_i.$\par
Assume that $i<j$.
Suppose first that $i=1, j\in\{2, 3\}$ or $i=2, j=4$.
We have that $a_i^m\,\rho\,a_i^{M_{ij}}$ by the above claim, 
so $(a_i^m, a_i^{M_{ij}})$ is a consequence of $X_i$, and likewise $(a_j^n, a_j^{N_{ij}})$ is a consequence of $X_j$.
Therefore, $(a_i^m, a_j^n)$ is a consequence of $X_i, X_j$ and $X_{ij}$.\par 
Now suppose that $i=1, j=4$ or $i=2, j=3$.
We have that $m=M_{ij}+r$ for some $r\geq 0$.
From the proof of (2) in the above claim, we obtain $a_j^{N_{ij}+r}$ from $a_i^m$ using $X_{ij}$.
Since $a_j^{N_{ij}+r}\,\rho\,a_j^n$, we have that $(a_j^{N_{ij}+r}, a_j^n)$ is a consequence of $X_j$.
Therefore, $(a_i^m, a_j^n)$ is a consequence of $X_{ij}$ and $X_j$.
\end{proof}

\begin{thm}
\label{mfp}
Let $M$ and $N$ be two monoids.
Then the monoid free product of $M$ and $N$ is right noetherian if and only if one of the following holds:
\begin{enumerate}
\item $M$ is right noetherian and $N$ is trivial, or vice versa;
\item both $M$ and $N$ contain precisely two elements.
\end{enumerate}
\end{thm}

\begin{proof}
Let $U$ be the monoid free product of $M$ and $N.$
If $N$ is trivial, then $U$ is isomorphic to $M,$ so we assume that both $M$ and $N$ are non-trivial.\par 
($\Rightarrow$) Suppose that $U$ is right noetherian, and yet $|N|\geq 3.$ 
Choose non-identity elements $a\in M$ and $b, c\in N.$
We consider two cases.\par
\textit{Case 1}: $a, b, c$ are units.
We have a free subgroup $G=\langle ab, ac\rangle$ of $U.$ 
Since $G$ has subgroups of infinite rank, it is not noetherian, contradicting Corollary \ref{subgroup}.\par
\textit{Case 2}: at least one of $a, b, c$ is not a unit,
and hence not a right unit by Corollary \ref{unit}.
For $i\in\mathbb{N},$ let $u_i=(ba)^icabac.$ 
Let $I$ be the right ideal of $U$ generated by the set $\{u_i : i\in\mathbb{N}\}.$
We claim that $I$ is not finitely generated, contradicting that $U$ is (weakly) right noetherian.
Indeed, if $I$ were finitely generated, then it could be generated by a finite set $$X=\{u_i : 1\leq i\leq k\}.$$
Let $i\in\{1, \dots, k\}.$  Since at least one of $a, b, c$ is not right invertible,
it is easy to see that, for every $u\in U,$ the first $2i+2$ free factors of $u_iu$ are $ b, a, \dots, b, a, c, a;$
in particular, the factor in position $2i+1$ is $c.$
On the other hand, the factor in position $2i+1$ of $u_{k+1}$ is $b.$
Therefore, there do not exist any $i\in\{1, \dots, k\}$ and $u\in U$ such that $u_{k+1}=u_iu,$ so $I$ is not generated by $X.$\par 
($\Rightarrow$) Now suppose that $|M|=|N|=2.$
There are two monoids with two elements, namely $\mathbb{Z}_2$ and $\{0, 1\}.$\par
If $M$ and $N$ are both isomorphic to $\mathbb{Z}_2,$
then it is well known that $U$ is isomorphic to a semidirect product of $\mathbb{Z}$ and $\mathbb{Z}_2,$
so $U$ is right noetherian by Theorem \ref{sdpgroups}.\par
If $M$ and $N$ are both isomorphic to $\{0, 1\}$, then $U$ is isomorphic to $V^{1}$ where $V$ is the semigroup free product of two trivial semigroups, 
so $U$ is right noetherian by Theorem \ref{sfp} and Corollary \ref{adjoinidentity}.\par 
Finally, suppose that $M=\mathbb{Z}_2$ and $N=\{0, 1\}$.
Then $U$ has the presentation
$$\langle a, b\,|\,a^2=1, b^2=b\rangle.$$
Let $u_1=ab$, $u_2=ba$ and $u_3=bab$, and let $U_i=\{u_i^n : n\geq 0\}$.
Also, let $U_4=\{u_1^ia : i\geq 1\}$.
We have that $U=\bigcup_{i=1}^4U_i$, and each $U_i$ is isomorphic to $\mathbb{N}_0$.\par
Let $\rho$ be a right congruence on $U$.
For each $i\in\{1, 2, 3\}$, let $\rho_i$ be the restriction of $\rho$ to $U_i$.
Since $U_i$ is right noetherian, $\rho_i$ is generated by some finite set $X_i$.
Let $i, j\in\{1, 2, 3\}$ with $i<j$.
If $u_i^m\,\rho\,u_j^n$ for some $m, n\in\mathbb{N}$, let
$$M_{ij}=\text{min}\{m\in\mathbb{N} : u_i^m\,\rho\,u_j^n\text{ for some }n\in\mathbb{N}\},\,N_{ij}=\text{min}\{n\in\mathbb{N} : u_i^{M_{ij}}\,\rho\,u_j^n\},$$
and let $X_{ij}=\bigl\{(u_i^{M_{ij}}, u_j^{N_{ij}})\bigr\}.$
If there do not exist $m, n\in\mathbb{N}$ such that $u_i^m\,\rho\,u_j^n$, let $X_{ij}=\emptyset.$\par 
Suppose that $u_1^m\,\rho\,u_1^na$ for some $m, n\geq 1.$
Let $M$ be minimal such that $u_1^M\,\rho\,u_1^na$ for some $n\geq 1,$
and let $N$ be minimal such that $u_1^M\,\rho\,u_1^Na.$
By multiplying on the right by $b$ and by $u_1$, we obtain $u_1^M\,\rho\,u_1^{N+1}$ and $u_1^{M+1}\,\rho\,u_1^N$ respectively.
In particular, we have that $u_1^M\,\rho\,u_1^{M+2k}$ for all $k\in\mathbb{N}_0.$  Now let
$$H=\bigl\{(u_1^M, u_1^{N+1}a), (u_1^{M+1}, u_1^Na)\bigr\}.$$
If there do not exist $m, n\in\mathbb{N}$ such that $u_1^m\,\rho\,u_1^na$, let $H=\emptyset$.\par
We claim that $\rho$ is generated by $X\cup H$ where
$$X=\Biggl(\bigcup_{i=1}^3X_i\Biggr)\cup\Biggl(\bigcup_{1\leq i<j\leq 3}X_{ij}\Biggr).$$
Let $u\,\rho\,v$ with $u\neq v$.
If $u, v\in\bigcup_{i=1}^3U_i$, then a similar argument to the one in Theorem \ref{sfp} proves that $(u, v)$ is a consequence of $X.$
Assume therefore that $v\in U_4$, so $v=u_1^na$ for some $n\in\mathbb{N}.$  
There are two cases to consider.\par
\textit{Case 1}: $u\in U_i, i\in\{2, 3, 4\}.$
We have that $ua\in\bigcup_{i=1}^3U_i,$ $va=u_1^n\in U_1$ and $ua\,\rho\,va.$
Therefore, $(ua, va)$ is a consequence of $X.$ 
Hence, multiplying on the right by $a,$ we have that $(u, v)$ is a consequence of $X.$\par
\textit{Case 2}: $u\in U_1.$
Since $u_1^M\,\rho\,u_1^{M+2k}$ for all $k\in\mathbb{N}_0,$ we have that $u\,\rho\,u_1^r$ where $r\in\{M, M+1\},$ 
so $(u, u_1^r)$ is a consequence of $X_1.$
Using $H,$ we obtain from $u_1^r$ the element $u_1^sa$ where $s\in\{N, N+1\}.$
Finally, by the same argument as in \textit{Case 1}, we have that $(u_1^sa, v)$ is a consequence of $X.$
\end{proof}

\begin{remark}
Since the group free product coincides with the monoid free product for groups, 
it follows from Theorem \ref{mfp} that the free product of two non-trivial groups is noetherian if and only if both groups are isomorphic to $\mathbb{Z}_2.$
\end{remark}

\noindent\textbf{Acknowledgement.}
The authors would like to thank two anonymous referees for their comments and suggestions, and, particularly, for a question that led to Proposition \ref{prop:Schutz}.

\end{document}